\documentclass[a4paper,10pt]{amsart}

\usepackage{amsmath}
\usepackage{amsthm}
\usepackage{amssymb}
\usepackage{amsfonts}
\usepackage{mathrsfs}  % for \mathscr
\usepackage{enumitem} % more enumerate/itemize options
\usepackage{tikz}
\usepackage{tikz-cd}
\usetikzlibrary{arrows.meta,calc}
\usepackage{xcolor}
\usepackage{float}
\usepackage[bookmarks=false,hyperindex,pdftex,colorlinks,citecolor=blue,urlcolor=cyan]{hyperref}
\usepackage[a4paper,lmargin=3cm,rmargin=3cm,tmargin=4cm,bmargin=4cm,marginparwidth=2.8cm,marginparsep=1mm]{geometry} % reduce side margins a bit
\DeclareMathOperator{\dist}{dist}                           % distance between sets
\DeclareMathOperator{\lspan}{span}                          % linear span
\DeclareMathOperator{\conv}{conv}                           % convex hull
\DeclareMathOperator{\diam}{diam}                           % diameter
\DeclareMathOperator{\Lip}{Lip}                             % Lipschitz functions
\newcommand{\RR}{\mathbb{R}}                                % real numbers
\newcommand{\HH}{\mathcal{H}}                               % Hausdorff measure
\newcommand{\ep}{\varepsilon}
\newcommand{\one}{\mathbf 1}

\newcommand{\abs}[1]{\left|{#1}\right|}                     % absolute value
\newcommand{\pare}[1]{\left({#1}\right)}                    % parentheses
\newcommand{\set}[1]{\left\{{#1}\right\}}                   % set by extension
\newcommand{\norm}[1]{\left\|{#1}\right\|}                  % norm
\newcommand{\duality}[1]{\left<{#1}\right>}                 % dual action
\newcommand{\cl}[1]{\overline{#1}}                          % closure
\newcommand{\wscl}[1]{\overline{#1}^{w^*}}                  % weak* closure
\newcommand{\lipfree}[1]{\mathcal{F}({#1})}                 % Lipschitz-free space
\newcommand{\lipnorm}[1]{\norm{#1}_L}                       % Lipschitz norm/constant
\newcommand{\ideal}[1]{\mathcal{I}({#1})}
\newcommand{\const}[1]{\mathcal{K}({#1})}

\renewcommand{\leq}{\leqslant}
\renewcommand{\geq}{\geqslant}

\theoremstyle{plain}
\newtheorem{theorem}{Theorem}[section]
\newtheorem{lemma}[theorem]{Lemma}

\newtheorem{proposition}[theorem]{Proposition}

\newtheorem*{claim*}{Claim}
\newtheorem{fact}[theorem]{Fact}

\theoremstyle{definition}
\newtheorem*{definition*}{Definition}
\newtheorem{definition}[theorem]{Definition}

\newtheorem{remark}[theorem]{Remark}

\hypersetup{pdftitle={Lipschitz-free spaces are strongly unique preduals},
  pdfauthor={Ramon J. Aliaga, Marek Cuth, Felipe Vico}}

\begin{document}
\title{Lipschitz-free spaces are strongly unique preduals}

\author[R. J. Aliaga]{Ram\'on J. Aliaga}
\address[R. J. Aliaga]{Instituto Universitario de Matem\'atica Pura y Aplicada,
Universitat Polit\`ecnica de Val\`encia,
Camino de Vera S/N, 46022 Valencia, Spain}
\email{ramon.aliaga@upv.es}

\author[M. C\'uth]{Marek C\'uth}
\address[M. C\'uth]{Charles University, Faculty of Mathematics and Physics,
Department of Mathematical Analysis,
Sokolovsk\'a 83, 186 75 Prague 8, Czech Republic}
\email{marek.cuth@matfyz.cuni.cz}

\author[F. Vico]{Felipe Vico}
\address[F. Vico]{Instituto de Telecomunicaciones y Aplicaciones Multimedia,
Universitat Polit\`ecnica de Val\`encia,
Camino de Vera S/N, 46022 Valencia, Spain}
\email{fevibon@teleco.upv.es}

%\date{}

\begin{abstract}
We prove that the Lipschitz-free space $\mathcal{F}(M)$ is the strongly unique isometric predual of $\mathrm{Lip}_0(M)$ for every metric space $M$, solving a longstanding open problem of Weaver.
The result is proved first for length metric spaces by analyzing the behavior of Lipschitz functions on paths in $M$ and the topological properties of subspaces of functions vanishing on certain barrier-sets of $M$.
Then, the result is extended to general metric spaces by embedding them into length spaces in such a way that the restriction operator for Lipschitz functions is weak$^*$ continuous with respect to their selected preduals.
Finally, we also give a counterexample to the codimension-one inheritance assertion used in Weaver's original proof for bounded and geodesic metric spaces.
\end{abstract}

\subjclass[2020]{Primary 46B10; Secondary 46E15, 54E50}
% 46B10: Duality and reflexivity in normed linear and Banach spaces
% 46E15: Banach spaces of continuous, differentiable or analytic functions
% 54E50: Complete metric spaces
\keywords{Unique predual, strongly unique predual, Lipschitz space, Lipschitz-free space,
length space, strongly exposed point.}
\maketitle
%\enlargethispage{2pt}

\section*{AI disclosure statement and note on the current version}

The main results in this document were obtained by OpenAI's GPT-5.6 Sol Pro and GPT-6 Pro (powered by GPT-6 Astra), with relatively limited mathematical input from the authors.

We initially presented the AI with the problem together with the relevant papers of Weaver. It identified a counterexample to one of Weaver's lemmas, and we then asked it to look for a corrected argument. The solution emerged iteratively: the AI first obtained proofs in several special cases, while we suggested which cases seemed useful to investigate next, and was eventually asked to synthesize these into a proof of the general result. No additional mathematical ideas from the authors were needed in this discovery process.

We subsequently examined the AI-generated arguments in detail, independently verified the proofs, and substantially reorganized, streamlined, and edited their presentation. We take full responsibility for the mathematical content of the manuscript. The proofs given here remain, in substance, those generated by the AI, but their present form reflects our careful verification and restructuring.

This is an early version of the manuscript; additional exposition and polishing will be added later. We have chosen to make it public already so that the main result can be used in further work and the proof can be independently examined by the community.

\section{Introduction}

Let $X$ be a dual Banach space. We say that $X$ has a \emph{unique predual} if there is a unique, up to linear isometry, Banach space $Y$ such that $Y^*$ is linearly isometric to $X$. We say that $X$ has a \emph{strongly unique predual} if moreover any linear isometry $T:X\to Y^*$ onto a dual Banach space is weak$^*$-to-weak$^*$ continuous, hence the adjoint of an isometry between the corresponding preduals. The latter is a formally stronger property, and it is currently unknown whether both properties are actually equivalent. Examples of spaces that are strongly unique preduals (i.e. such that their duals have strongly unique preduals) include spaces with the Radon-Nikod\'ym property \cite{Godefroy_predual}; preduals of von Neumann algebras, including the usual $L_1$-spaces \cite{Godefroy_predual}; and separable L-embedded spaces \cite{Pfitzner_x}. The sequence space $\ell_1$ is maybe the simplest example of a dual Banach space failing to have a unique predual.

In this paper, we consider the unique predual problem for Lipschitz spaces, i.e. Banach spaces of real-valued Lipschitz functions on a metric space $M$. This problem goes back, at least, to the first edition of Weaver's monograph \cite{Weaver2} in 1999. The theory of Lipschitz spaces has many analogies with the theory of von Neumann algebras \cite{AP22,Weaver2,Weaver96_3}, and so one could expect them to have strongly unique preduals as well. Our main result confirms the validity of that conjecture.

\begin{theorem}\label{thm:main}
For every metric space $M$, the Lipschitz-free space
$\lipfree{M}$ is the strongly unique isometric predual of the
Lipschitz space $\Lip_0(M)$.
\end{theorem}

Prior to this paper, the main advances in the unique predual problem had been obtained by Weaver. In \cite{Weaver18}, Weaver proved strong uniqueness for spaces $\Lip(M)$ of bounded Lipschitz functions \cite[Theorem~2.4]{Weaver18}, with a proof reminiscent of Sakai's original proof for von Neumann algebras \cite{Sakai}. He then deduced Theorem~\ref{thm:main} for the case where $M$ is bounded by means of a lemma, attributed to U.~Bader, asserting that the property of having a strongly unique predual passes to weak$^*$ closed subspaces of codimension one \cite[Lemma~3.1]{Weaver18} (in the bounded case, $\Lip_0(M)$ is a weak$^*$ closed, $1$-codimensional subspace of some $\Lip(M')$). Further, Weaver provided a geometric argument proving universal weak$^*$ closedness of subspaces of Lipschitz functions vanishing on certain balls under the assumption that $M$ is geodesic, which allowed him to reduce the problem to the bounded case and thereby deduce Theorem~\ref{thm:main} for geodesic metric spaces $M$ as well \cite[Theorem~3.3]{Weaver18}.

Unfortunately, in 2026 M.~Gonz\'alez reported a gap in the proof of \cite[Lemma~3.1]{Weaver18} that called all results for $\Lip_0$ spaces into question (see \cite{Weaver_corrigendum} and the updated arXiv preprint \cite{Weaver_arxiv}, where the issue is acknowledged). It remained unclear whether the statement of \cite[Lemma~3.1]{Weaver18} itself was correct and its proof could be amended. In fact, \cite[Lemma~3.1]{Weaver18} turns out to be false: we construct an explicit counterexample in the appendix (see Theorem~\ref{thm:counterexample}).

While \cite[Lemma~3.1]{Weaver18} is false, the rest of Weaver's arguments in \cite{Weaver18} remain valid and they serve as the starting point for our proof of Theorem~\ref{thm:main}. This includes his reduction to the bounded case (see Remark~\ref{rem:bounded-reduction}), as well as his argument for universal weak$^*$ closedness for balls, which we first extend from the geodesic case to the general case (Proposition~\ref{prop:balls}).

We first prove Theorem~\ref{thm:main} under the assumption that $M$ is a length space (Theorem~\ref{thm:length}). To do this, we extend Weaver's universal weak$^*$ closedness argument to certain ``cut-sets'' of $M$ (Lemma~\ref{lem:boundary}), and then relate the structure of Lipschitz spaces over those sets with that of $L^\infty$ spaces (see Lemma~\ref{lem:cut-projection2}). Thus, we ultimately reduce the unique predual problem to that of $L^\infty$ spaces, which is known to have a positive answer.

The general case of Theorem~\ref{thm:main} is then proved by embedding the metric space $M$ into a length space $N$, obtained by adding geodesic segments between selected pairs of points in $M$ in such a way that preduals of $\Lip_0(M)$ can be related to preduals of $\Lip_0(N)$ (Proposition~\ref{prop:compatible-predual}), thereby reducing the problem to the length case. This part depends crucially on the characterization of strongly exposed points of $B_{\lipfree{M}}$ obtained by Garc\'ia-Lirola, Proch\'azka and Rueda Zoca in \cite{GLPRZ}.

The rest of the paper is structured as follows. After this introduction, Section~\ref{sec:prelim} establishes the notation and some necessary facts related to preduals, Lipschitz and Lipschitz-free spaces, and length metric spaces. Section~\ref{sec:subspaces} extends Weaver's original arguments from balls to more general subsets of $M$. Theorem~\ref{thm:main} is then proved, first in Section~\ref{sec:length} for the particular case of length spaces, then in Section~\ref{sec:general} in general. Finally, the appendix provides a counterexample to \cite[Lemma~3.1]{Weaver18}.

\section{Preliminaries}
\label{sec:prelim}

\subsection{Preduals}

In what follows, all scalar fields are real, and $B_X$ and $S_X$ denote the closed unit ball and unit sphere of a Banach space $X$.
If $X$ is a dual Banach space, we may identify preduals of $X$ as subspaces of its dual $X^*$ via the canonical inclusion of every Banach space in its bidual. We use the following notation to make that explicit.

\begin{definition}
A \emph{concrete predual} of a dual Banach space $X$ is a closed subspace
$Y\subset X^*$ such that
\[
 X\longrightarrow Y^*,\qquad x\longmapsto [y\mapsto y(x)]
\]
is a surjective linear isometry. A subset of $X$ is
\emph{universally weak$^*$ closed} if it is $\sigma(X,Y)$-closed
for every concrete predual $Y$ of $X$. We similarly define universally weak$^*$ compact sets, universally weak$^*$ continuous mappings, etc.
\end{definition}

The following standard fact restates strong uniqueness in terms of concrete preduals.
We include a proof for the reader's convenience.

\begin{fact}\label{fact:concrete-preduals}
A concrete predual $E\subset X^*$ is strongly unique if and only if
it is the only concrete predual of $X$. These conditions are also
equivalent to $E\subset Y$ for every concrete predual $Y$ of $X$.
\end{fact}

\begin{proof}
Every concrete predual is norming. Consequently, concrete preduals
cannot be properly nested: if $E\subsetneq Y$ were both concrete
preduals, the Hahn-Banach theorem would give a nonzero $x\in X=Y^*$ vanishing
on $E$, a contradiction. This proves the equivalence
of the last two conditions.

Now let $F$ be a Banach space and let $T:X\to F^*$ be a surjective
linear isometry. Define
\[
 j_T:F\longrightarrow X^*,\qquad j_T(f)(x)=(Tx)(f).
\]
Then 
\[
 \norm{j_T(f)}
 =\sup_{x\in B_X}|(Tx)(f)|
 =\norm f,
\]
so $j_T(F)$ is a closed subspace of $X^*$. It is a concrete predual:
for its evaluation map
\[
 J:X\longrightarrow j_T(F)^*,\qquad
 J(x)(j_T(f))=(Tx)(f),
\]
we have
\[
 \norm{Jx}=\sup_{f\in B_F}|(Tx)(f)|=\norm x.
\]
Moreover, if $\Phi\in j_T(F)^*$, then $\Phi\circ j_T\in F^*$, so
$\Phi\circ j_T=Tx$ for some $x\in X$, and hence $\Phi=Jx$. Thus $J$
is a surjective linear isometry.

Finally,
\[
 T:(X,\sigma(X,E))\longrightarrow(F^*,\sigma(F^*,F))
\]
is continuous if and only if, for every $f\in F$, the functional
\[
 x\longmapsto (Tx)(f)=j_T(f)(x)
\]
belongs to $E$, that is, if and only if $j_T(F)\subset E$.

If $E$ is the only concrete predual, then $j_T(F)=E$ for every such
$T$, and hence $E$ is strongly unique. Conversely, let $E$ be strongly
unique and let $Y\subset X^*$ be a concrete predual. Applying the
preceding observation to its evaluation isometry
\[
 T_Y:X\longrightarrow Y^*,\qquad T_Y(x)(y)=y(x),
\]
gives $j_{T_Y}(Y)\subset E$. But
\[
 j_{T_Y}(y)(x)=T_Y(x)(y)=y(x),
\]
so $j_{T_Y}(Y)=Y$. Thus $Y\subset E$, and non-nesting gives $Y=E$.
\end{proof}

Recall that a point $x\in B_X$ is a \emph{strongly exposed point}
of $B_X$ if there exists $f\in B_{X^*}$ with $f(x)=1$
such that, for every sequence $(x_n)$ in $B_X$,
$f(x_n)\to1$ implies $\norm{x_n-x}\to0$. We then say that $f$ \emph{strongly exposes}
$x$. An elementary equivalent formulation, recalled in
\cite[Section~4]{GLPRZ}, is that, for every $\ep>0$,
there exists $\alpha>0$ such that
\[
 \{u\in B_X:f(u)>1-\alpha\}\subset x+\ep B_X.
\]
We shall need the following observation.

\begin{lemma}\label{lem:exposed}
Let $X$ be a Banach space. Every strongly exposed point of $B_X$ belongs to every concrete predual $Y\subset X^{**}$ of $X^*$.
\end{lemma}

\begin{proof}
Suppose $x\in B_X$ is strongly
exposed by $f\in X^*$, normalized
so that $\norm f=f(x)=1$. Given $\ep>0$, choose $\alpha>0$ with
\[
 \{v\in B_X:f(v)>1-\alpha\}\subset x+\ep B_X.
\]
Goldstine's theorem and weak$^*$ closedness of
$x+\ep B_{X^{**}}$ imply
\[
 \{v^{**}\in B_{X^{**}}:v^{**}(f)>1-\alpha\}
 \subset x+\ep B_{X^{**}}.
\]
Indeed, a net from $B_X$ converging weak$^*$ to such a $v^{**}$ is
eventually in the displayed strict slice. Since $Y$ is norming for
$X^*$, this slice contains a member of $B_Y$. Thus
$\dist(x,Y)\leq\ep$. The norm closedness of $Y$ proves $x\in Y$.
\end{proof}

An \emph{$M$-projection} on a Banach space $X$ is a bounded
projection $P$ satisfying
\[
 \norm x=\max\{\norm{Px},\norm{(I-P)x}\}\qquad(x\in X),
\]
where $I$ is the identity operator.
The range of $P$ is called an \emph{$M$-summand}. An \emph{$L$-projection}
is defined by replacing the maximum with a sum, and its range is an
\emph{$L$-summand}. The adjoint of an $M$-projection
is an $L$-projection. In particular, $PX$ is an \emph{$M$-ideal}
in $X$: its annihilator is the $L$-summand $(I-P)^*X^*$.
See \cite[Section~I.1]{HWW} for further reference.

\begin{lemma}\label{lem:Mprojection}
Let $P$ be an $M$-projection on a dual Banach space $X$.
Then $P$ is universally weak$^*$ continuous.
If $U=PX$ has a strongly unique concrete predual $U_*\subset U^*$,
then every concrete predual $Y$ of $X$ satisfies
\begin{equation}\label{eq:Mprojection}
Y=P^*Y\oplus_1(I-P)^*Y
\quad,\quad
P^*Y=\{\lambda\circ P:\lambda\in U_*\}
\quad\text{and}\quad
(I-P)^*Y=Y\cap U^\perp.
\end{equation}
\end{lemma}

\begin{proof}
Fix a concrete predual $Y$ of $X$ and identify $X=Y^*$. The projection $P$ is weak$^*$ continuous by
\cite[Theorem~I.1.9]{HWW}, hence $P^*Y\subset Y$. Since $P^*$ is an $L$-projection, we get the first and last equalities in \eqref{eq:Mprojection}.

Put $E=P^*Y$ and let
$V=\{\eta|_U:\eta\in E\}\subset U^*$.
Every $\eta\in E$ satisfies $\eta=\eta\circ P$.
Since $P(B_X)=B_U$, restriction
$R:E\to V$, $R\eta=\eta|_U$, is a surjective linear isometry.
In particular, $V$ is norm closed in $U^*$.

We verify that $V$ is a concrete predual of $U$.
For $u\in U$, we have
\[
 \sup_{\lambda\in B_V}|\lambda(u)|
 =\sup_{\eta\in B_E}|\eta(u)|
 =\sup_{y\in B_Y}|y(u)|
 =\norm u,
\]
where the second equality follows from
$(P^*y)(u)=y(u)$ and $P^*(B_Y)=B_E$.
Thus the evaluation map $J:U\to V^*$ is an isometry.
To see that it is onto, take $\varphi\in V^*$.
By Hahn--Banach, $\varphi\circ R\in E^*$ extends to a
functional on $Y$, represented by some $x\in Y^*=X$.
For every $\eta\in E$,
\[
 \varphi(R\eta)=\eta(x)=\eta(Px)=(R\eta)(Px).
\]
Hence $\varphi=J(Px)$, proving surjectivity.

Strong uniqueness and Fact~\ref{fact:concrete-preduals} now give
$V=U_*$. Since $\eta=(R\eta)\circ P$ for every $\eta\in E$,
we obtain
$E=\{\lambda\circ P:\lambda\in U_*\}$,
finishing the proof of \eqref{eq:Mprojection}.
\end{proof}

\subsection{Lipschitz and Lipschitz-free spaces}

Let $M$ be a metric space with metric $d$. The closed ball with center $x\in M$ and radius $r>0$ is denoted $B(x,r)$. We assume tacitly that $M$ is a \emph{pointed} metric space, meaning that we have chosen a distinguished point $0\in M$ as a base point for our definitions. Then we define the \emph{Lipschitz space} over $M$ as
$$
\Lip_0(M) = \set{f:M\to\RR \,:\, \text{$f$ is Lipschitz and $f(0)=0$}} .
$$
This becomes a Banach space when endowed with the Lipschitz norm given by
$$
\lipnorm{f} = \sup\set{\frac{f(x)-f(y)}{d(x,y)} \,:\, x\neq y\in M}
$$
(note that, if the restriction $f(0)=0$ is lifted, $\lipnorm{\cdot}$ is merely a seminorm). If a different base point is chosen, the resulting Lipschitz space is linearly isometric to $\Lip_0(M)$ via the mapping $f\mapsto f-f(0)$.

Lipschitz spaces are dual Banach spaces. The simplest way to construct their canonical preduals is as follows. Consider the evaluation functionals $\delta_M(x)\in\Lip_0(M)^*$, $x\in M$, given by $f\mapsto f(x)$. These functionals generate a subspace
$$
\lipfree{M} = \cl{\lspan}\set{\delta_M(x)\,:\,x\in M} ,
$$
of $\Lip_0(M)^*$, called the \emph{Lipschitz-free space} over $M$. The space $\lipfree{M}$ is always a concrete predual of $\Lip_0(M)$; hence, if $\Lip_0(M)$ has a strongly unique predual, then that predual must be $\lipfree{M}$. The weak$^*$ topology induced by $\lipfree{M}$ on $\Lip_0(M)$ agrees on norm-bounded subsets of $\Lip_0(M)$ with the topology of pointwise convergence. For further reference on Lipschitz and Lipschitz-free spaces, we direct the reader to Weaver's monograph \cite{Weaver2} (where $\lipfree{M}$ is denoted $\text{\AE}(M)$ and called ``Arens-Eells space'' instead).

Since every Lipschitz function on $M$ can be uniquely extended to its completion $\cl{M}$ with the same Lipschitz constant, the Lipschitz space $\Lip_0(M)$ can be identified with $\Lip_0(\cl{M})$, and similarly for $\lipfree{M}$ and $\lipfree{\cl{M}}$. In particular, it suffices to prove Theorem~\ref{thm:main} under the assumption that $M$ is complete. Thus, in what follows we will usually assume that the ambient space $M$ is complete.

The map $\delta_M:M\to\lipfree{M}$ is an isometric embedding (we will usually omit the subscript $M$ and write simply $\delta$ when there is no risk of confusion). As a result, given two distinct points $x,y\in M$, the \emph{(elementary) molecule}
$$
m_{xy} = \frac{\delta(x)-\delta(y)}{d(x,y)}
$$
belongs to $S_{\lipfree{M}}$. The unit ball $B_{\lipfree{M}}$ is the norm-closed convex hull of the set of all molecules on $M$. Molecules will be relevant in connection to Lemma \ref{lem:exposed} through the following result by Garc\'ia-Lirola, Proch\'azka and Rueda Zoca \cite{GLPRZ}.

\begin{definition}
A pair of distinct points $x,y\in M$ has \emph{property~\textup{(Z)}} if,
for every $\ep>0$, there is $z\in M\setminus\{x,y\}$ such that
\begin{equation}\label{eq:Z}
 d(x,z)+d(z,y)-d(x,y)
 \leq\ep\min\{d(x,z),d(z,y)\}.
\end{equation}
A metric space has property~\textup{(Z)} if all its distinct pairs of points do.
\end{definition}

\begin{theorem}[{\cite[Theorem~5.4]{GLPRZ}}]
\label{thm:GLPRZ}
The strongly exposed points of $B_{\lipfree{M}}$ are precisely the molecules $m_{xy}$ corresponding to pairs of distinct points $x,y\in M$ that fail property~\textup{(Z)}.
\end{theorem}

We will also need the following estimation for the distance between two molecules in $\lipfree{M}$. A similar computation can also be found in \cite[Lemma~1.2]{Veeorg}.

\begin{lemma}\label{lm:distance of molecules}
Let $p,q,u,v\in M$ with $p\neq q$ and $u\neq v$. Then
$$
\norm{m_{uv}-m_{pq}} \leq 2\,\frac{d(u,p)+d(v,q)}{d(u,v)} .
$$
\end{lemma}

\begin{proof}
The identity
$$
m_{uv}-m_{pq} = \frac{\delta(u)-\delta(p)-\delta(v)+\delta(q)}{d(u,v)} + \frac{d(p,q)-d(u,v)}{d(u,v)}\,m_{pq}
$$
gives
$$
\norm{m_{uv}-m_{pq}} \leq \frac{d(u,p)+d(v,q)+\abs{d(p,q)-d(u,v)}}{d(u,v)} \leq \frac{2(d(u,p)+d(v,q))}{d(u,v)} .
$$
\end{proof}

Let $N$ be a subset of $M$. By McShane's theorem \cite[Theorem~1.33]{Weaver2}, every Lipschitz function $f:N\to\RR$ can be extended to a function $F:M\to\RR$ without increasing its Lipschitz constant. As a consequence, if $0\in N$ then $\lipfree{N}$ can be identified isometrically with the subspace
$$
\mathcal{F}_M(N) = \cl{\lspan}\set{\delta_M(x) \,:\, x\in N}
$$
of $\lipfree{M}$ (see \cite[Theorem~3.7]{Weaver2}), and we will usually do so by omitting the subscript $M$ whenever it is clear from context. We also consider the following subspaces of $\Lip_0(M)$:
\begin{align*}
\ideal{N} &= \set{ f\in\Lip_0(M) \,:\, \text{$f(x)=0$ for all $x\in N$} }, \\
\const{N} &= \set{ f\in\Lip_0(M) \,:\, \text{$f$ is constant on $N$}}.
\end{align*}
To be more precise we should write $\mathcal{I}_M(N)$ and $\mathcal{K}_M(N)$, but again the ambient metric space $M$ will usually be fixed and so we will omit it for simplicity. While these spaces may be different in general, we clearly have $\const{N}=\ideal{N}$ if $0\in N$. In that case, we moreover have $\ideal{N}=\lipfree{N}^\perp$ in $\Lip_0(M)$, and restriction identifies $\Lip_0(M)/\ideal{N}$ isometrically with $\Lip_0(N)$.

Isometries between Lipschitz spaces induced by changes of base point
transport concrete preduals and
universal weak$^*$ closedness assertions. For $w\in W\subset M$,
we write $\Lip_w(W)$ for the Lipschitz functions on $W$ vanishing
at $w$, and identify its canonical predual
$\mathcal{F}_w(W)$ with the closed span of
$\{\delta_M(x)-\delta_M(w):x\in W\}$ in $\lipfree{M}$.

\begin{lemma}\label{lem:restriction}
Let $Y$ be a concrete predual of $\Lip_0(M)$, and let $W\subset M$ be
nonempty. Suppose that $\const{W}$ is $\sigma(\Lip_0(M),Y)$-closed.
Then $\Lip_w(W)$ admits a concrete predual $J$ such that
$R_W^*(J)\subset Y$, where
\[
 R_W:\Lip_0(M)\longrightarrow\Lip_w(W),\qquad
 R_Wf=f|_W-f(w),\quad w\in W.
\]
    In particular, $R_W$ is continuous from
$\sigma(\Lip_0(M),Y)$ to $\sigma(\Lip_w(W),J)$.
\end{lemma}

\begin{proof}
The kernel of $R_W$ is precisely $\const{W}$. Moreover, $R_W$ is an isometric quotient. Indeed, let $g\in\Lip_w(W)$. By the McShane extension
theorem, there is a Lipschitz extension $G:M\to\mathbb R$ of $g$ with $\norm G_{\Lip}=\norm g$. Setting $f=G-G(0)$ gives $f\in\Lip_0(M)$, $\norm f_{\Lip}=\norm g$, and $R_Wf=g$. Hence $R_W$ induces an isometric isomorphism
\[
\Lip_0(M)/\const{W}\cong\Lip_w(W).
\]

Put $V=Y\cap\const{W}^{\perp}$.
Since $\const{W}$ is weak$^*$ closed in $\Lip_0(M)=Y^*$, the bipolar
theorem gives $\const{W}=V^\perp$. Consequently, the usual duality between quotients and annihilators
identifies $\Lip_0(M)/\const{W}$ isometrically with $V^*$. Combining this identification with
$\Lip_0(M)/\const{W}\cong\Lip_w(W)$, we obtain a concrete predual of
$\Lip_w(W)$. More explicitly, for $v\in V$ define $j_v\in\Lip_w(W)^*$ by
\[
 j_v(R_Wf)=v(f),\qquad f\in\Lip_0(M),
\]
and set $J=\{j_v:v\in V\}$. The preceding quotient identification shows that $J$ is a concrete
predual of $\Lip_w(W)$. Finally,
\[
 (R_W^*j_v)(f)=j_v(R_Wf)=v(f),
\]
so $R_W^*(J)=V\subset Y$. Finally, for each $j\in J$, the functional $j\circ R_W=R_W^*j$ belongs to $Y$ and is therefore $\sigma(\Lip_0(M),Y)$-continuous, which proves the ``In particular'' part.
\end{proof}

\subsection{Differentiability}

In what follows, $\lambda$ denotes Lebesgue measure and $\HH^1$ denotes $1$-Hausdorff measure. All almost-everywhere statements refer to Lebesgue measure or its restriction.

Let $E$ be a Borel subset of $\RR$. Recall that, by Rademacher's theorem and the fundamental theorem of calculus, any Lipschitz function $g:E\to\RR$ is differentiable at almost every point of $E$, with $\norm{g'}_\infty\leq\lipnorm{g}$. If $E=[a,b]$ is an interval with $a<b$, then $\norm{g'}_\infty=\lipnorm{g}$ and $\int_a^bg'(t)\,dt=g(b)-g(a)$ (see e.g. \cite[Corollary~1.39]{Weaver2}). All of this applies, in particular, to $g=f\circ\gamma$ when $\gamma:E\to M$ is a Lipschitz map and $f\in\Lip_0(M)$.

We also need to recall some notions about metric differentiation. These notions were introduced by Kirchheim in \cite{Kirchheim}, but we also reference \cite[Section~1.3]{AGPP} for a simpler exposition. Given a Lipschitz map $\gamma:E\to M$, its \emph{metric differential} at a non-isolated point $t\in E$ is
\[
 \operatorname{MD}(\gamma,t)
 =\lim_{\substack{s\to t\\s\in E\setminus\{t\}}}
   \frac{d(\gamma(s),\gamma(t))}{|s-t|},
\]
whenever this limit exists; see
\cite[Definition~1.7]{AGPP}.

\begin{fact}\label{fact:metric-area}
\hfill
\begin{enumerate}[label=\textup{(\roman*)}]
\item The metric differential $\operatorname{MD}(\gamma,t)$
exists for almost every $t\in E$. Moreover, if $E=[a,b]$,
then for almost every $t\in(a,b)$,
\[
 d(\gamma(t+s),\gamma(t+s'))
 =
 \operatorname{MD}(\gamma,t)|s-s'|
 +o(|s|+|s'|)
 \qquad\text{when $s,s'\to 0$.}
\]

\item (Area formula) We have
\[
 \int_E\operatorname{MD}(\gamma,t)\,d\lambda(t)
 =
 \int_{\gamma(E)}
   \#\gamma^{-1}(x)\,d\HH^1(x) .
\]
In particular,
\[
 \operatorname{MD}(\gamma,t)=0
 \quad\text{for almost every }t\in E
 \quad\Longleftrightarrow\quad
 \HH^1(\gamma(E))=0.
\]
\end{enumerate}
\end{fact}

\begin{proof}[Sketch of the proof]
Part (i) follows from Kirchheim's theorem \cite[Theorem~2]{Kirchheim}.
To apply the theorem in its stated form, embed $M$ isometrically into a Banach space $X$ and extend $\gamma$ to a Lipschitz map $\widetilde\gamma:\RR\to X$ with the same
Lipschitz constant; see \cite[Remark~1.10]{AGPP}.
For almost every $t$, Kirchheim's theorem provides a seminorm $N_t$ on $\RR$ such that
\[
 \norm{\widetilde\gamma(t+s)-\widetilde\gamma(t+s')}
 =N_t(s-s')+o(|s|+|s'|).
\]
Writing $N_t(h)=c_t|h|$, taking $s'=0$, and letting
$s\to0$ with $t+s\in E$ identifies
$c_t=\operatorname{MD}(\gamma,t)$ at every non-isolated
point $t\in E$ where the expansion holds.
The isolated points of $E$ form a countable set,
so this proves almost-everywhere existence.
For an interval domain, the same identification gives
the displayed two-point expansion.

The area formula (ii) is \cite[Theorem~1.9]{AGPP} or \cite[Corollary~8]{Kirchheim}; by (i), $\operatorname{MD}(\gamma,t)$ exists almost everywhere so the left integral is well defined.
\end{proof}

Finally, we isolate two elementary facts about real Lipschitz functions.

\begin{fact}\label{fact:real}\hfill
\begin{enumerate}[label=\textup{(\roman*)}]
\item If two real-valued Lipschitz functions defined on subsets of $\RR$ agree on a measurable set,
then their derivatives agree almost everywhere on that set.
\item Let $H$ and $v$ be real-valued Lipschitz functions defined on intervals, with
the range of $v$ in the domain of $H$. If $H'=0$ almost everywhere
on a Borel set $E$, then $(H\circ v)'=0$ almost everywhere on
$v^{-1}(E)$.
\end{enumerate}
\end{fact}

\begin{proof}
For (i), let $E$ be the set of points where the two functions $f,g$
agree and both are differentiable. This has full measure in
the set where they agree. Since the isolated points of $E$ form a countable set, almost every $x\in E$ is non-isolated. For each such $x$, choose $x_n\in E\setminus{x}$ with $x_n\to x$. Writing $h=f-g$, we obtain
\[
 h'(x)=\lim_{n\to\infty}\frac{h(x_n)-h(x)}{x_n-x}=0,
\]
since $h$ vanishes on $E$. Thus $f'=g'$ almost everywhere
on the set where they agree.

For (ii), first let $w$ be any real-valued Lipschitz function
on an interval and let $B$ be a Borel subset of that
interval. At every non-isolated point $t\in B$ where
$w$ is differentiable, we have
$\operatorname{MD}(w|_B,t)=|w'(t)|$.
Such points have full measure in $B$, since the isolated
points of $B$ form a countable set.
Fact~\ref{fact:metric-area}(ii), together with the
equality of $\HH^1$ and $\lambda$ on $\RR$, therefore gives
\[
 w'=0\text{ almost everywhere on }B
 \quad\Longleftrightarrow\quad
 \lambda(w(B))=0.
\]
Applying this equivalence with $w=H$ and $B=E$ yields
$\lambda(H(E))=0$. Since
$(H\circ v)(v^{-1}(E))\subset H(E)$, the image of
$v^{-1}(E)$ under $H\circ v$ is also null.
A second application, with $w=H\circ v$ and
$B=v^{-1}(E)$, proves (ii).
\end{proof}

\subsection{Length spaces}

A \emph{path} in $M$ is a continuous map $\gamma:[a,b]\to M$.
Its \emph{length} is
\[
 L(\gamma)=
 \sup\left\{
   \sum_{i=1}^{n}d\bigl(\gamma(t_{i-1}),\gamma(t_i)\bigr):
   a=t_0<t_1<\cdots<t_n=b
 \right\},
\]
and $\gamma$ is \emph{rectifiable} if $L(\gamma)<\infty$. The path $\gamma$ is \emph{parametrized by arclength}
if $L(\gamma|_{[u,v]})=v-u$ whenever $a\leq u\leq v\leq b$.
Such a path is $1$-Lipschitz.

We recall the following standard fact about arclength
parametrization; for the reader's convenience, we refer to
\cite[Proposition~2.5.9]{BBI}. Every rectifiable path
$\gamma:[a,b]\to M$ of length $L>0$ factors as
$\gamma=\widetilde\gamma\circ s$, where
$s:[a,b]\to[0,L]$ is a continuous nondecreasing surjection
and $\widetilde\gamma:[0,L]\to M$ is parametrized by
arclength. We call $\widetilde\gamma$ an
\emph{arclength parametrization} of $\gamma$.
It has the same image and endpoints as $\gamma$, it is
injective whenever $\gamma$ is injective, and it is always $1$-Lipschitz.
A path of length zero is constant.

A complete metric
space $M$ is called a \emph{length space} if, for every
$x,y\in M$ and $\delta>0$, there is a
rectifiable path $\gamma$ joining $x$ to $y$ with
$L(\gamma)<d(x,y)+\delta$. Equivalently, for every $x,y\in M$ and $\ep>0$,
there exists $z\in M$ such that
\[
 \max\{d(x,z),d(z,y)\}\leq\tfrac12d(x,y)+\delta.
\]
This equivalence is standard;
for the reader's convenience, we refer to
\cite[Lemma~2.4.10 and Theorem~2.4.16(2)]{BBI}\footnote{To compare the midpoint conventions, the triangle inequality gives
$\min\{d(x,z),d(z,y)\}\geq d(x,y)/2-\delta$.
Thus our midpoint condition agrees with the one used there
after rescaling the error.}.

We shall need the following characterization, due to Avil\'es and Mart\'inez-Cervantes \cite{AMC}.

\begin{theorem}[\cite{AMC}]\label{thm:AMC}
A complete metric space is a length space if and only if it has
property~\textup{(Z)}.
\end{theorem}

    We collect some elementary consequences of the length-space
property that will be used below.

\begin{fact}\label{fact:length-basic}
Let $M$ be a length space.
\begin{enumerate}[label=\textup{(\roman*)}]
\item
Every open ball is rectifiably path connected.
Every connected component $C$ of an open set $U\subset M$
is open and rectifiably path connected, and
$\partial C\subset\partial U$.
Moreover, $\partial C\ne\varnothing$ whenever $C\ne M$.

\item
If $F\subset M$ meets every path from $x$ to $y$, then
\[
 d(x,y)=\inf_{z\in F}\bigl(d(x,z)+d(z,y)\bigr).
\]

\item
If $F_1,F_2\subset M$ are nonempty closed sets covering $M$,
and $f:M\to\RR$ is Lipschitz on each $F_i$, then $f$ is
Lipschitz and
\[
 \lipnorm{f}
 =\max\{\lipnorm{f|_{F_1}},\lipnorm{f|_{F_2}}\}.
\]

\item
For every real Lipschitz function $f$ on $M$,
\[
 \lipnorm{f}
 =\sup_\sigma\norm{(f\circ\gamma)'}_\infty,
\]
where $\gamma:[0,L]\to M$ runs over all nonconstant
rectifiable paths parametrized by arclength.
The supremum is understood to be zero if $M$ is a singleton.
\end{enumerate}
\end{fact}

\begin{proof}
For \textup{(i)}, first consider an open ball of radius $r$.
Each of its points can be joined to the center by a path
of length less than $r$, which therefore stays in the ball.
Thus every open ball is rectifiably path connected.

Next, let $U\subset M$ be open. The equivalence classes
under joinability by rectifiable paths in $U$ are open,
since each point of $U$ has an open ball contained in $U$.
Each class $C$ is connected, and its complement in $U$
is a union of other open classes. Thus $C$ is both open
and closed in $U$. If $K$ is the connected component
of $U$ containing $C$, then $C$ is a nonempty subset
of $K$ that is both open and closed in $K$.
Connectedness gives $C=K$. Hence the connected components
of $U$ are open and rectifiably path connected.

Finally, such a component $C$ is relatively closed in $U$
and open in $M$, so its boundary misses $U$.
Since $\cl{C}\subset\cl{U}$, this gives
$\partial C\subset\partial U$.
If $\partial C=\varnothing$, then $C$ is also closed in $M$;
connectedness of $M$ therefore implies $C=M$.

For \textup{(ii)}, a path from $x$ to $y$ of length
less than $d(x,y)+\ep$ meets $F$ at some $z$, and hence
$d(x,z)+d(z,y)<d(x,y)+\ep$.
Let $\ep\downarrow0$ and use the triangle inequality.

For \textup{(iii)}, only pairs
$x\in F_1\setminus F_2$ and $y\in F_2\setminus F_1$
need consideration. Every path between them meets
$F_1\cap F_2$, so \textup{(ii)} and the Lipschitz bounds
on the two sets give the assertion.

In \textup{(iv)}, each $\gamma$ is $1$-Lipschitz, giving
the inequality $\geq$. For the reverse inequality, denote the right-hand side by $K$.
Given $x\ne y$ and $\ep>0$, choose an arclength-parametrized
path $\gamma:[0,L]\to M$ from $x$ to $y$ with
$L<d(x,y)+\ep$. Then
\[
 |f(y)-f(x)|
 =\left|\int_0^L(f\circ\gamma)'(t)\,dt\right|
 \leq KL
 \leq K\bigl(d(x,y)+\ep\bigr).
\]
Letting $\ep\downarrow0$ and taking the supremum over
distinct $x,y$ gives $\lipnorm{f}\leq K$.
\end{proof}

\begin{lemma}\label{lem:arc}
Distinct points of a length space can be joined by an injective
rectifiable path $\gamma:[0,L]\to M$ parametrized by arclength.
For almost every $t\in(0,L)$, such a path satisfies
\begin{equation}\label{eq:metric-speed}
 d(\gamma(t+s),\gamma(t+s'))
 =|s-s'|+o(|s|+|s'|) \qquad\text{when }s,s'\to 0.
\end{equation}
\end{lemma}

\begin{proof}
Let $M$ be a length space and $p\neq q\in M$.
Start with a rectifiable path joining $p$ and $q$, and let $C$ be its compact image. Let $L>0$ be the infimum of the lengths of rectifiable paths in $C$ joining $p$ to $q$. Using the arclength parametrizations recalled above, choose such paths $\gamma_n:[0,L_n]\to C$, parametrized by arclength, with $L_n\to L$. The maps $\alpha_n:[0,1]\to C$ defined by $\alpha_n(t)=\gamma_n(L_nt)$ are $L_n$-Lipschitz. Since $(L_n)$ is bounded and $C$ is compact, the Arzel\`a-Ascoli theorem gives a subsequence converging uniformly to a path $\alpha:[0,1]\to C$ with endpoints $p,q$. For each fixed partition, the corresponding sums of distances converge, so the definition of length gives
\[
 L(\alpha)\leq\liminf_{n\to\infty}L(\alpha_n)=L.
\]
The definition of $L$ gives the reverse inequality; hence $L(\alpha)=L$. Let $\gamma:[0,L]\to C$ be the arclength parametrization of $\alpha$. If $\gamma(s)=\gamma(t)$ for some $s<t$, concatenating $\gamma|_{[0,s]}$ and $\gamma|_{[t,L]}$ would give a path in $C$ joining $p$ to $q$ of length $s+(L-t)=L-(t-s)<L$, a contradiction. Thus $\gamma$ is injective.

For the second assertion, let $\gamma:[0,L]\to M$ be any injective rectifiable path parametrized by arclength. By Fact~\ref{fact:metric-area}(i), its metric differential $v(t)=\operatorname{MD}(\gamma,t)$ exists almost everywhere. Since $\gamma$ is $1$-Lipschitz, we have $0\leq v\leq1$ almost everywhere. By injectivity and Fact~\ref{fact:metric-area}(ii),
\[
 \int_0^L v(t)\,dt
 = \int_{\gamma([0,L])} 1\,d\HH^1(x)
 =\HH^1(\gamma([0,L]))
 =L,
\]
where the last equality uses the standard identity between the length of an injective rectifiable path and the $\HH^1$-measure of its image. Consequently, $v=1$ almost everywhere, and Fact~\ref{fact:metric-area}(i) gives
\eqref{eq:metric-speed}.
\end{proof}

\section{Universally weak$^*$ closed subspaces}
\label{sec:subspaces}

The following variant of Weaver's ball argument used in
\cite[Theorem~3.3]{Weaver18} uses open balls and applies to
arbitrary metric spaces.

\begin{proposition}\label{prop:balls}
If $M$ is a metric space and $B$ is a nonempty open ball in $M$,
then $\const{B}$ is universally weak$^*$ closed in $\Lip_0(M)$.
\end{proposition}

\begin{proof}
Let $B$ be the open ball with center $p\in M$ and radius $R>0$. For $0<r<R$, set
\[
 h_r(x)=\min\{r,d(x,p)\}-\min\{r,d(0,p)\},\qquad
 t_r=1-\frac rR>0.
\]
We claim that
\begin{equation}\label{eq:ball-intersection}
 B_{\const{B}}=B_{\Lip_0(M)}\cap\bigcap_{0<r<R}
 \{f\in\Lip_0(M):\lipnorm{h_r+t_rf}\leq1,
                       \ \lipnorm{h_r-t_rf}\leq1\}.
\end{equation}
If $f$ belongs to the right side and $x\in B$, choose
$d(p,x)<r<R$. Then
\[
 d(p,x)\pm t_r(f(x)-f(p))
 =(h_r\pm t_rf)(x)-(h_r\pm t_rf)(p)
 \leq d(p,x).
\]
Thus $f(x)=f(p)$, so $f\in B_{\const{B}}$.

Conversely, let $f\in B_{\const{B}}$. Fix $0<r<R$ and $x,y\in M$, and let us estimate
\begin{equation}\label{eq:prop-balls-estimate}
|(h_r\pm t_rf)(y)-(h_r\pm t_rf)(x)| .
\end{equation}
If $x,y\in B$, $f(x)=f(y)$ and \eqref{eq:prop-balls-estimate} is $|h_r(y)-h_r(x)|\leq d(x,y)$; if $x,y\in M\setminus B$, $h_r(x)=h_r(y)$ and \eqref{eq:prop-balls-estimate} is $t_r|f(y)-f(x)|\leq d(x,y)$. For $x\in B$, $y\notin B$,
put $a=d(p,x)$. If $a\geq r$, then $h_r(x)=h_r(y)$.
If $a<r$, then
\[
 |h_r(y)-h_r(x)|=r-a
 \leq\frac rR(R-a)\leq\frac rR d(x,y),
\]
and hence
\[
 |(h_r\pm t_rf)(y)-(h_r\pm t_rf)(x)|
 \leq\left(\frac rR+t_r\right)d(x,y)=d(x,y).
\]
This proves \eqref{eq:ball-intersection}.

Note that $\lipnorm{h_r\pm t_rf}\leq 1$ precisely when $f\in t_r^{-1}(B_{\Lip_0(M)}\mp h_r)$. Thus the right-hand side of \eqref{eq:ball-intersection} is an intersection of translates of multiples of $B_{\Lip_0(M)}$, hence universally weak$^*$ compact,
and the Krein--\v{S}mulyan theorem proves the assertion.
\end{proof}

\begin{remark}\label{rem:bounded-reduction}
Since Proposition~\ref{prop:balls} requires no extra condition on $M$, Weaver's bounded reduction argument from \cite[Theorem~3.1]{Weaver_arxiv} applies in general. That is, if Theorem~\ref{thm:main} holds for bounded metric spaces, then it holds for all metric spaces. Indeed, let $M$ be any pointed metric space and let $Y$ be a concrete
predual of $\Lip_0(M)$. For $n=1,2,\ldots$, let $B_n$ be the open ball with center $0$ and radius $n$.
Proposition~\ref{prop:balls} and Lemma~\ref{lem:restriction} give
a concrete predual $J_n$ of $\Lip_0(B_n)$ such that
$R_{B_n}^*(J_n)\subset Y$. Assuming the bounded case of Theorem~\ref{thm:main},
we obtain $J_n=\lipfree{B_n}$. Since
$R_{B_n}^*\delta_{B_n}(x)=\delta_M(x)$ for $x\in B_n$, and
$M=\bigcup_n B_n$, every $\delta_M(x)$, $x\in M$ belongs to $Y$.
Fact~\ref{fact:concrete-preduals} completes the proof.
\end{remark}

Proposition~\ref{prop:balls} can be easily extended to more general domains.

\begin{proposition}\label{prop:domains}
If $M$ is a metric space and $U\subset M$ is nonempty, open and
connected, then $\const{U}$ is universally weak$^*$ closed in $\Lip_0(M)$.
\end{proposition}

\begin{proof}
This follows from Proposition~\ref{prop:balls} by observing that
$$
\const{U} = \bigcap\set{\const{B} \,:\, \text{$B$ is an open ball contained in $U$} } .
$$
Indeed, inclusion $\subset$ is obvious. For the converse, note that any $f\in\bigcap_B\const{B}$ is locally constant at every point of $U$ because $U$ is open. Thus level sets $f^{-1}(c)\cap U$ are simultaneously open and closed in $U$ and, since $U$ is connected, $f$ must be constant in $U$.
\end{proof}

\begin{lemma}
\label{lem:boundary}
Let $M$ be a length space, let $U\subset M$ be nonempty, open and
connected, and let $B$ be a nonempty connected component of $M\setminus\cl{U}$.
Put $W=\partial B$, $A=M\setminus\cl{B}$.
Then $A$ and $W$ are nonempty, and $\const{A}$, $\const{B}$ and $\const{W}$ are
universally weak$^*$ closed.
\end{lemma}

\begin{proof}
By Fact~\ref{fact:length-basic}(i), every component of
$M\setminus\cl{U}$ is open and has nonempty boundary
contained in $\cl{U}$.
Moreover $U\subset A$, in particular $A$ is nonempty.

Let us prove that
\begin{equation}\label{eq:component-kernels}
 \const{A}=\bigcap\{\const{V}:V\text{ is a component of }A\}.
\end{equation}
Suppose $f\in \Lip_0(M)$ is constant on each component of $A$. Since $U$
is connected, $f$ is constant, say equal to $c$, on $U$ and hence on
$\cl{U}$. Any component $B'$ of $M\setminus\cl{U}$ different from $B$ is disjoint from
$\cl{B}$, as it is open and disjoint from $B$. Thus $B'\subset A$,
so $f$ is constant on $B'$. Since $\partial B'\subset\cl{U}$ is non-empty, that constant must be $c$. Finally, $A$ is contained in
the union of $\cl{U}$ and these other components, so we conclude that $f\in\const{A}$ as claimed.
Each component $V$ of $A$ is open by
Fact~\ref{fact:length-basic}(i), and connected.
Proposition~\ref{prop:domains} proves universal closedness
of $\const{A}$ and also of $\const{B}$.

There is a disjoint partition $M=A\cup W\cup B$. Furthermore,
$W\subset\cl{A}\cap\cl{B}$, because $W\subset\cl{U}$
and $U\subset A$.
Consequently, $\cl{A}$ and $\cl{B}$ form a closed cover
of $M$ with $\cl{A}\cap\cl{B}=W$.

We shall show that
\begin{equation}\label{eq:sum of balls}
 B_{\const{W}}=B_{\const{A}}+B_{\const{B}}.
\end{equation}
This will be enough as $B_{\const{A}}$ and $B_{\const{B}}$ are universally weak$^*$ compact, so the same holds for their sum, thus $\const{W}$ is universally weak$^*$ closed by the Krein-\v{S}mulyan theorem.

Assume that the base point $0$ belongs to $\cl{A}$; the proof is similar if $0\in\cl{B}$ instead. Given $f\in B_{\const{W}}$, let $c$ be the constant value of $f$ at $W$ and define functions $g$ and $h$ on $M$ by
$$
g(x) = \begin{cases}
0 &\text{, if $x\in A$} \\
f(x)-c &\text{, if $x\notin A$}
\end{cases}
\qquad\text{and}\qquad
h(x) = \begin{cases}
c &\text{, if $x\in B$} \\
f(x) &\text{, if $x\notin B$}
\end{cases}
.
$$
It is clear that $g$ is constant on $\cl{A}\supset A\cup W$, $h$ is constant on $\cl{B}\supset B\cup W$, $g(0)=h(0)=0$, and $f=g+h$. We may also check that $\lipnorm{g}\leq 1$. Indeed, we have $g(x)-g(y)=0$ if $x,y\in\cl{A}$ and $\abs{g(x)-g(y)}=\abs{f(x)-f(y)}\leq d(x,y)$ if $x,y\in\cl{B}$. Now assume $x\in A$ and $y\in B$. By Fact~\ref{fact:length-basic}(ii), for every $\ep>0$ there exists $w\in W$ such that $d(x,w)+d(w,y)\leq d(x,y)+\ep$, and
$$
\abs{g(x)-g(y)} = \abs{f(y)-c} = \abs{f(y)-f(w)} \leq d(w,y) \leq d(x,y)+\ep .
$$
Letting $\ep\to 0$ shows $\lipnorm{g}\leq 1$. Similar reasoning yields $\lipnorm{h}\leq 1$. This establishes inclusion $\subset$ in \eqref{eq:sum of balls}.

Conversely, let $g\in B_{\const{A}}$ and $h\in B_{\const{B}}$, and set $f=g+h$. Clearly $f$ is constant on $W\subset\cl{A}\cap\cl{B}$. Again, we have $\abs{f(x)-f(y)}=\abs{h(x)-h(y)}\leq d(x,y)$ if $x,y\in\cl{A}$ and $\abs{f(x)-f(y)}=\abs{g(x)-g(y)}\leq d(x,y)$ if $x,y\in\cl{B}$. For $x\in A$ and $y\in B$ and any $\ep>0$, use Fact~\ref{fact:length-basic}(ii) to obtain $w\in W$ such that $d(x,w)+d(w,y)\leq d(x,y)+\ep$, then
\begin{align*}
\abs{f(x)-f(y)} &\leq \abs{f(x)-f(w)} + \abs{f(w)-f(y)} \\
&= \abs{h(x)-h(w)} + \abs{g(w)-g(y)} \\
&\leq d(x,w) + d(w,y) \leq d(x,y)+\ep .
\end{align*}
We conclude that $\lipnorm{f}\leq 1$ and $f\in B_{\const{W}}$. This completes the proof of \eqref{eq:sum of balls}.
\end{proof}

\section{Proof for length spaces}
\label{sec:length}

We will now prove Theorem~\ref{thm:main} under the assumption that $M$ is a complete length space. Our proof will be based on analyzing the behavior of Lipschitz functions along rectifiable paths in $M$, in particular in the ``cut-sets'' separating the endpoints from each other. For the next three lemmas, we fix a complete length space $M$ and two distinct points $p,q\in M$. Write $D=d(p,q)$ and define
\begin{equation}\label{eq:cutset}
 S=\{p,q\}\cup
 \{z\in M\setminus\{p,q\}:p,q\text{ belong to different
 components of }M\setminus\{z\}\}.
\end{equation}

\begin{lemma}\label{lem:cutset}
The set $S$ is compact, and $z\mapsto d(p,z)$ is an isometry from
$S$ onto a compact set $E\subset[0,D]$ containing $0$ and $D$. 
Let $\iota:E\to S$ is its inverse.
There exists a $1$-Lipschitz function $r:M\to[0,D]$
such that $r\circ\iota=\operatorname{id}_E$ and, for
every connected $V\subset M\setminus S$, the interior of the
interval $r(V)$ is disjoint from $E$.
\end{lemma}

\begin{proof}
Every point of $S$ lies on every path from $p$ to $q$,
so $S$ is contained in the compact image of one such path.
To prove that $S$ is closed, fix $z\notin S$.
Then $p,q$ belong to the same component of the open set
$M\setminus\{z\}$. By Fact~\ref{fact:length-basic}(i),
there is a path from $p$ to $q$ avoiding $z$.
Its compact image misses an open ball around $z$.
Since that image contains $S$, this ball is disjoint
from $S$. Thus $S$ is closed, hence compact.

Applying Fact~\ref{fact:length-basic}(ii) with the
separating set $F=\{z\}$ gives
\[
 D=d(p,z)+d(z,q)\qquad(z\in S).
\]
For distinct $z,w\in S$, either every path from $p$ to $w$
meets $z$, or every path from $w$ to $q$ meets $z$.
Indeed, otherwise we could concatenate two paths avoiding
$z$ to obtain a path from $p$ to $q$ avoiding $z$.
Another application of Fact~\ref{fact:length-basic}(ii)
therefore gives
\[
 d(p,w)=d(p,z)+d(z,w)
 \quad\text{or}\quad
 d(w,q)=d(w,z)+d(z,q).
\]
Together with the identity for $D$, either equality implies
\[
 d(z,w)=|d(p,z)-d(p,w)|.
\]
Hence $z\mapsto d(p,z)$ is an isometry of $S$ onto a
compact set $E\subset[0,D]$ containing $0$ and $D$.

Define
\begin{equation}\label{eq:cut-coordinate}
 r(x)=\frac{d(x,p)-d(x,q)+D}{2}
\end{equation}
for $x\in M$.
The triangle inequality gives $0\leq r\leq D$ and
$\lipnorm{r}\leq1$. Moreover, the identity for $D$ gives
$r(z)=d(p,z)$ for $z\in S$, so $r(\iota(s))=s$ for every $s\in E$.

Finally, let $V\subset M\setminus S$ be connected.
By continuity, $r(V)$ is an interval.
Suppose that $s\in E$ lies in its interior.
Since $r(V)\subset[0,D]$, we have $0<s<D$.
Put $z=\iota(s)\in S$. As $V$ is connected and avoids $z$,
it lies in a component $C$ of $M\setminus\{z\}$.
The points $p,q$ belong to different components,
so at least one of them lies outside $C$.
If $q\notin C$, every path from $x\in V$ to $q$ meets $z$.
Fact~\ref{fact:length-basic}(ii) and the triangle inequality
therefore give
\[
 d(x,q)=d(x,z)+D-s,\qquad
 d(x,p)\leq d(x,z)+s
 \qquad(x\in V).
\]
Substituting into \eqref{eq:cut-coordinate} yields
$r(x)\leq s$ for every $x\in V$.
If $p\notin C$ instead, a symmetric argument gives
$r(x)\geq s$ for every $x\in V$.
In either case, $s$ cannot lie in the interior of $r(V)$,
a contradiction.
\end{proof}

\begin{lemma}\label{lem:cut-projection2}
There exists an $M$-projection $P$ on $\Lip_0(M)$ such that
\[
 P^*(\lipfree{M})\subset Y
 \qquad\text{for every concrete predual $Y$ of $\Lip_0(M)$}
\]
and, moreover, for every $f\in\Lip_0(M)$ and every nonconstant
rectifiable path $\gamma:[0,L]\to M$ parametrized by
arclength,
\begin{equation}\label{eq:cut-projection-derivative}
 (Pf\circ\gamma)'(t)=(f\circ\gamma)'(t)
 \quad\text{for almost every $t\in [0,L]$ such that $\gamma(t)\in S$.}
\end{equation}
\end{lemma}

\begin{proof}
Let $E,\iota,r$ be as in Lemma~\ref{lem:cutset}.
For $f\in\Lip_0(M)$, let $H_f$ be the extension of $f\circ\iota:E\to\RR$ to $[0,D]$ by affine interpolation on each gap in $E$.\footnote{Explicitly, on each interval $[a,b]$ with $a<b$ and $[a,b]\cap E=\{a,b\}$ we put
\[
 H_f(t)=
 \frac{b-t}{b-a}f(\iota(a))
 +\frac{t-a}{b-a}f(\iota(b))
 \qquad(a<t<b).
\]}
This extension satisfies $\lipnorm{H_f}\leq\lipnorm{f}$ and the map $f\mapsto H_f$ is linear.

Equip $E$ with restricted Lebesgue measure and define linear operators $\mathcal R:\Lip_0(M)\to L_\infty(E)$ and $\mathcal T:L_\infty(E)\to\Lip_0(M)$ by
\begin{align}
 \mathcal Rf&=H_f'|_E,
 \label{eq:cut-R}\\
 (\mathcal Tg)(x)&=
 \int_0^{r(x)}\one_E(t)g(t)\,dt
 -\int_0^{r(0)}\one_E(t)g(t)\,dt,
 \qquad x\in M.
 \label{eq:cut-T}
\end{align}
Here $g\in L^\infty(E)$ is extended by zero outside $E$, and $0$ in
$r(0)$ denotes the base point of $M$.
Both $\mathcal R$ and $\mathcal T$ are non-expansive.
For $s\in E$, the definitions and the identity
$r\circ \iota=\operatorname{id}_E$ give
\[
 H_{\mathcal Tg}(s)
 =(\mathcal Tg)(\iota(s))
 =\int_0^s\one_E(t)g(t)\,dt
  -\int_0^{r(0)}\one_E(t)g(t)\,dt.
\]
The second integral is independent of $s$ and therefore does not affect derivatives. By Fact~\ref{fact:real}(i) and the fundamental theorem of calculus, $(H_{\mathcal Tg})'=g$ almost everywhere on $E$, so $\mathcal R\mathcal T=I$.
Thus $\mathcal T$ is an isometry, and $P=\mathcal T\mathcal R$ is a projection on $\Lip_0(M)$ (with $P=0$ if $E$ is null).

Let $U=\mathcal P(\Lip_0(M))=\mathcal T(L^\infty(E))$, and let us check that every $u\in U$ is locally constant on $M\setminus S$. Suppose $u=\mathcal Tg$ for $g\in L^\infty(E)$.
Since $S$ is closed, Fact~\ref{fact:length-basic}(i) shows that every point in $M\setminus S$ has a connected open neighborhood $V$ disjoint from $S$.
By Lemma~\ref{lem:cutset}, the interior of the interval $r(V)$ does not intersect $E$.
Thus, for any $x,y\in V$, the integrand $\one_Eg$ vanishes almost everywhere between $r(x)$ and $r(y)$, and
\[
 u(x)-u(y)=
 (\mathcal Tg)(x)-(\mathcal Tg)(y)
 =\int_{r(x)}^{r(y)}\one_E(t)g(t)\,dt=0.
\]
Hence $u$ is constant on $V$, as required.

Put $Q=I-P$. Since $\mathcal R\mathcal T=I$, we have
$\mathcal RQ=0$.
Fix $f\in\Lip_0(M)$.
Let $\gamma:[0,L]\to M$ be any nonconstant
rectifiable path parametrized by arclength and set $A=\gamma^{-1}(S)$. Outside $A$, the function $Pf\circ\gamma$ is locally constant. On $A$, the function $Qf\circ\gamma$ agrees
with $H_{Qf}\circ r\circ\gamma$.
Indeed, for $t\in A$ we have $\gamma(t)\in S$ and hence $\iota(r(\gamma(t)))=\gamma(t)$. Since $H_{Qf}$ agrees with $(Qf)\circ\iota$ on $E$,
it follows that
\[
 H_{Qf}(r(\gamma(t)))
 =Qf(\iota(r(\gamma(t))))
 =Qf(\gamma(t))
 .
\]
Since $H_{Qf}'|_E=\mathcal RQf=0$, Fact~\ref{fact:real}(ii) applied to $H_{Qf}$ and
$r\circ\gamma$ gives
\[
 (H_{Qf}\circ r\circ\gamma)'=0
 \quad\text{almost everywhere on }A,
\]
because $r(\gamma(A))\subset E$.
The functions $Qf\circ\gamma$ and
$H_{Qf}\circ r\circ\gamma$ agree on $A$, so
Fact~\ref{fact:real}(i) shows that their derivatives
also agree almost everywhere on $A$.
Thus
$(Qf\circ\gamma)'=0$ almost everywhere on $A$. Consequently,
\[
 (Pf\circ\gamma)'=\one_A(f\circ\gamma)',\qquad
 (Qf\circ\gamma)'=
 \one_{[0,L]\setminus A}(f\circ\gamma)'
 \quad\text{almost everywhere}.
\]
This proves \eqref{eq:cut-projection-derivative}.
The disjoint derivative supports and Fact~\ref{fact:length-basic}(iv) give
\begin{align*}
 \lipnorm{f}
 =\sup_\gamma
   \max\{\norm{(Pf\circ\gamma)'}_\infty,
          \norm{(Qf\circ\gamma)'}_\infty\}=\max\{\lipnorm{Pf},\lipnorm{Qf}\},
\end{align*}
where the supremum runs over these paths.
Thus $P$ is an $M$-projection.

Finally, $\mathcal T$ identifies
$U$ isometrically with $L_\infty(E)$, whose
canonical predual $L_1(E)$ is strongly unique.
Let $U_*\subset U^*$ be the strongly unique concrete predual of $U$.
For every concrete predual $Y$ of $\Lip_0(M)$, Lemma~\ref{lem:Mprojection} gives
\[
 P^*Y=\{\lambda\circ P:\lambda\in U_*\},
\]
which is independent of $Y$.
In particular, considering the canonical predual $\lipfree{M}$ gives us
\[
 P^*(\lipfree{M})=P^*Y\subset Y,
\]
as required.
\end{proof}

\begin{lemma}\label{lem:detour}
Let $Y$ be a concrete predual of $\Lip_0(M)$, and
$\Phi\in \Lip_0(M)^{**}$ a functional vanishing on $Y\subset \Lip_0(M)^*$.
Put $f(x)=\Phi(\delta(x))$, $x\in M$. If $\gamma:[0,L]\to M$ is an injective rectifiable path from $p$ to $q$ parametrized by arclength then
\[
 (f\circ\gamma)'(t)=0
 \quad\text{for almost every $t\in[0,L]$ such that $\gamma(t)\notin S$.}
\]
\end{lemma}

\begin{proof}
The function $f$ belongs to $\Lip_0(M)$ and $\lipnorm{f}\leq\norm\Phi$.
Fix $t\in(0,L)$ at which $f\circ\gamma$ is differentiable and
\eqref{eq:metric-speed} holds. Such points form a set of full
measure. We will prove that $(f\circ\gamma)'(t)=0$ for such $t$ assuming that $z=\gamma(t)\notin S$. So assume this,
then $p,q$ lie in the same component of $M\setminus\{z\}$.
By Fact~\ref{fact:length-basic}(i),
there is a path from $p$
to $q$ avoiding $z$, whose compact image will serve as a fixed
detour.

Fix $0<\ep<1/8$. By \eqref{eq:metric-speed}, there is
$0<h_0<\min\{t,L-t\}$ such that, for $|s|,|s'|\leq h_0$,
\begin{equation}\label{eq:linear-control}
 d(\gamma(t+s),\gamma(t+s'))
 \geq |s-s'|-\ep(|s|+|s'|).
\end{equation}
For $0<h<h_0/(1+4\ep)$, set
\[
 p_h=\gamma(t-h),\qquad q_h=\gamma(t+h),\qquad
 D_h=d(p_h,q_h),\qquad R_h=(2+\ep)h.
\]
The arclength upper bound and \eqref{eq:linear-control} give
\begin{equation}\label{eq:Dh}
 2(1-\ep)h\leq D_h\leq2h<R_h.
\end{equation}
Set
\begin{align*}
O_h &= \{x\in M:\max\{d(x,p_h),d(x,q_h)\}<R_h\} \\
L_h &= \{x\in M:\max\{d(x,p_h),d(x,q_h)\}\leq R_h\}
\end{align*}
and let $U_h$ be the component of $O_h$
containing $\gamma([t-h,t+h])$;
this image is contained in $O_h$ because $\gamma$ is
$1$-Lipschitz and $2h<R_h$.
The set $U_h$ is open and connected by Fact~\ref{fact:length-basic}(i).

\begin{figure}[H]
\centering
\begin{tikzpicture}[x=1.2cm,y=.72cm,font=\small,
                    line cap=round,line join=round]
 % A schematic planar lens. The small endpoint gaps are exaggerated.
 \pgfmathsetmacro{\lensangle}{acos(1/2.3)}
 \path[fill=black!7,draw=black!65,line width=.7pt]
  (0,{sqrt(2.3*2.3-1)})
  arc[start angle=\lensangle,end angle=-\lensangle,radius=2.3]
  arc[start angle=180+\lensangle,end angle=180-\lensangle,radius=2.3]
  --cycle;
 \draw[thin] (-4,0)--(4,0);
 \draw[line width=1.4pt] (-1,0)--(1,0);
 \draw[thick,densely dashed] (-4,0)
  ..controls (-3.5,3.3) and (3.5,3.3)..(4,0);
 \node at (0,2.88) {fixed detour};
 \node at (0,.8) {$U_h$};
 \node at (-2.65,1.25) {$B_h$};
 \draw[thin,black!65] (.77,-1.48)--(1.45,-1.15)--(2.1,-1.15);
 \node[right] at (2.1,-1.15) {$W_h=\partial C_h$};
 \foreach \x in {-4,-2.05,-1.3,-1,0,1,1.3,2.05,4}
  \fill (\x,0) circle[radius=1.4pt];
 \node[below=4pt] at (-4,0) {$p$};
 \node[below=4pt] at (4,0) {$q$};
 \node[below=4pt] at (-2.05,0) {$x_h^-$};
 \node[below=4pt] at (2.05,0) {$x_h^+$};
 \node[above left=3pt] at (-1.3,0) {$a_h$};
 \node[above right=3pt] at (1.3,0) {$b_h$};
 \node[below right=4pt] at (-1,0) {$p_h$};
 \node[below left=4pt] at (1,0) {$q_h$};
 \node[below=4pt] at (0,0) {$z$};
 \node[below=4pt] at (3.1,0) {$\gamma$};
\end{tikzpicture}
\caption{Illustration of the construction.}
\label{fig:detour}
\end{figure}
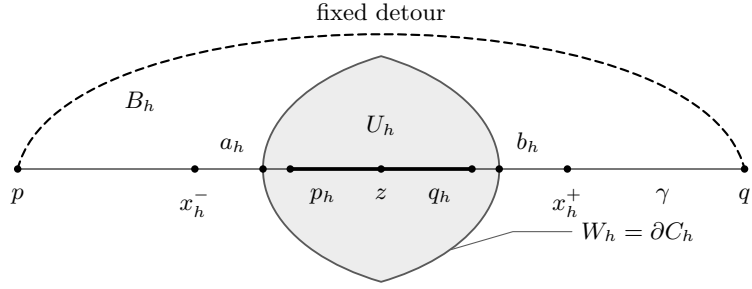

We claim that the outer points
\[
 x_h^-=\gamma(t-(1+4\ep)h),\qquad
 x_h^+=\gamma(t+(1+4\ep)h)
\]
can be joined outside the closed set $L_h$ when $h$ is small enough.
The construction is illustrated in Figure~\ref{fig:detour}.
For $(1+4\ep)h\leq|s|\leq h_0$, applying
\eqref{eq:linear-control} with $s'=-h$ and $s'=h$ gives
\begin{align*}
 \max\{d(\gamma(t+s),p_h),d(\gamma(t+s),q_h)\}
 &\geq\max\{|s+h|,|s-h|\}-\ep(|s|+h)=(1-\ep)(|s|+h)\\
 &\geq(1-\ep)(2+4\ep)h>R_h
\end{align*}
where the last inequality follows from
$(1-\ep)(2+4\ep)-(2+\ep)=\ep-4\ep^2>0$.
Therefore the portions of $\gamma$ between $x_h^-$ and $\gamma(t-h_0)$ and between $x_h^+$ and $\gamma(t+h_0)$ lie outside $L_h$. On the other hand,
the two portions of $\gamma$ with parameter outside
$(t-h_0,t+h_0)$ have compact images avoiding $z$, by injectivity.
Their union with the fixed detour also stays a positive distance
from $z$. Thus they will be outside $L_h$ for sufficiently small $h$, as
\[
 L_h\subset B(z,(3+\ep)h)
\]
since $d(z,p_h)\leq h$. To summarize, if $h$ is small enough then the path
from $x_h^-$ backwards along $\gamma$ to $p$, along the detour to
$q$, and backwards along the final part of $\gamma$ to $x_h^+$
does not intersect $L_h$. This proves the claim.

Since $\cl{U_h}\subset L_h$, $x_h^+$ and $x_h^-$ lie in the same
component $C_h$ of $M\setminus\cl{U_h}$. Put $W_h=\partial C_h$.
By Lemma~\ref{lem:boundary}, $\const{W_h}$ is weak$^*$ closed
with respect to $Y$. We also have
\begin{equation}\label{eq:Wh-level}
 \max\{d(w,p_h),d(w,q_h)\}=R_h\qquad \text{for all }w\in W_h.
\end{equation}
Indeed, Fact~\ref{fact:length-basic}(i) and the inclusion
$\partial\cl{U_h}\subset\partial U_h$ give
\[
 W_h\subset\partial(M\setminus\cl{U_h})=\partial\cl{U_h}
 \subset\partial U_h\subset\partial O_h
\]
and \eqref{eq:Wh-level} follows from the continuity of $x\mapsto\max\{d(x,p_h),d(x,q_h)\}$.

Fix $w_0\in W_h$ and define for $w\in W_h$
\[
 \psi(w)=
 \left(\frac{D_h}2-\frac34d(w,p_h)\right)_+
 -
 \left(\frac{D_h}2-\frac34d(w,q_h)\right)_+,
 \qquad F(w)=\psi(w)-\psi(w_0),
\]
where $r_+=\max\{r,0\}$.
We claim that $F\in B_{\Lip_{w_0}(W_h)}$ and
\begin{equation}\label{eq:lens-affine-bound}
 \norm{m_{uv}-m_{p_hq_h}}
 \leq8\bigl(1-F(m_{uv})\bigr)
 \qquad \text{for $u,v\in W_h$, $\ u\ne v$}.
\end{equation}

The two positive parts defining $\psi$ are
$3/4$-Lipschitz and, by \eqref{eq:Wh-level} and $R_h\geq D_h$,
cannot both be positive at a point of $W_h$.
Thus $\psi$ is $3/4$-Lipschitz on each of the sets
$\{\psi\geq0\}$ and $\{\psi\leq0\}$.
Consequently, unless $\psi(u)>0>\psi(v)$, we have
$F(m_{uv})\leq3/4$, and \eqref{eq:lens-affine-bound}
follows from $\norm{m_{uv}-m_{p_hq_h}}\leq2$.

Suppose now that $\psi(u)>0>\psi(v)$.
Then $d(u,p_h),d(v,q_h)<2D_h/3$, so \eqref{eq:Wh-level}
gives $d(u,q_h)=d(v,p_h)=R_h$.
Put $t=d(u,v)$ and $c=d(u,p_h)+d(v,q_h)$, then
the triangle inequality yields
\[
 t\geq R_h-\min\{d(u,p_h),d(v,q_h)\}
   \geq D_h-\frac c2.
\]
Since $\psi(u)-\psi(v)=D_h-3c/4$, this inequality implies
\[
 1-F(m_{uv})
 =\frac{t-D_h+3c/4}{t}
 \geq\frac{c}{4t}.
\]
In particular, $F(m_{uv})\leq1$, and Lemma~\ref{lm:distance of molecules} gives
\[
 \norm{m_{uv}-m_{p_hq_h}}
 \leq\frac{2c}{t}
 \leq8\bigl(1-F(m_{uv})\bigr).
\]
Considering all possible pairs $(u,v)$ proves that $F$ is
$1$-Lipschitz and establishes \eqref{eq:lens-affine-bound}.

The restriction of $\gamma$ to
$[t-(1+4\ep)h,t-h]$ joins $x_h^-\in C_h$ to
$p_h\in U_h$ and therefore meets $W_h$ at some point $a_h$.
The restriction to $[t+h,t+(1+4\ep)h]$ similarly meets
$W_h$ at some point $b_h$, which is distinct from $a_h$ because $\gamma$ is injective. We have
$$
d(a_h,p_h),d(b_h,q_h)\leq 4\ep h<\frac{2}{3}D_h
$$
so $\psi(a_h)>0$ and $\psi(b_h)<0$. Putting $\rho=d(a_h,p_h)+d(b_h,q_h)\leq 8\ep h$,
direct evaluation gives
\[
 F(m_{a_hb_h})
 =\frac{D_h-3\rho/4}{d(a_h,b_h)}
 \geq\frac{D_h-3\rho/4}{D_h+\rho}
 =1-\frac{7\rho}{4(D_h+\rho)}\geq 1-\frac{7\rho}{4D_h}
\geq 1-\frac{56\ep h}{8(1-\ep)h} > 1-8\ep .
\]

Lemma~\ref{lem:restriction} provides a concrete predual
$J$ of $\Lip_{w_0}(W_h)$ with $R_{W_h}^*(J)\subset Y$.
Since $J$ is $1$-norming, we may choose $\zeta\in B_J$ with
$\zeta(F)>1-8\ep$.
For every $g\in B_{\Lip_0(M)}$ and distinct $u,v\in W_h$,
\eqref{eq:lens-affine-bound} gives
\[
 (R_{W_h}g+8F)(m_{uv})
 \leq g(m_{p_hq_h})+\norm{m_{uv}-m_{p_hq_h}}+8F(m_{uv})
 \leq g(m_{p_hq_h})+8.
\]
Taking the supremum over all ordered pairs $u\ne v$
in $W_h$, we obtain
$\lipnorm{R_{W_h}g+8F}\leq8+g(m_{p_hq_h})$.
Consequently,
\[
 (R_{W_h}^*\zeta-m_{pq})(g)
 =\zeta(R_{W_h}g+8F)-8\zeta(F)-g(m_{p_hq_h})\leq8\bigl(1-\zeta(F)\bigr)<64\ep.
\]
Taking the supremum over $g\in B_{\Lip_0(M)}$ yields
$\norm{R_{W_h}^*\zeta-m_{p_hq_h}}\leq 64\ep$.
Finally, because $\Phi$ vanishes on $R_{W_h}^*\zeta\in Y$ and $D_h\leq2h$,
\[
 \frac{|f(q_h)-f(p_h)|}{2h}
 =\frac{D_h}{2h}\,|\Phi(m_{p_hq_h})|
 \leq|\Phi(m_{p_hq_h})|
 \leq\norm\Phi\,\norm{R_{W_h}^*\zeta-m_{p_hq_h}}
 \leq\norm\Phi\,64\ep.
\]
Since $f\circ\gamma$ is differentiable at $t$, letting $h\to 0$ and then $\ep\to 0$ yields $(f\circ\gamma)'(t)=0$.
\end{proof}

\begin{theorem}\label{thm:length}
If $M$ is a complete length space, then $\lipfree{M}$ is the strongly
unique isometric predual of $\Lip_0(M)$.
\end{theorem}

\begin{proof}
Let $Y\subset\Lip_0(M)^*$ be any concrete predual of $\Lip_0(M)$;
by Fact~\ref{fact:concrete-preduals}, it suffices to prove
$\lipfree{M}\subset Y$.
By Hahn--Banach, this amounts to showing that every
$\Phi\in\Lip_0(M)^{**}$ vanishing on $Y$ also vanishes
on $\lipfree{M}$.
Define $f(x)=\Phi(\delta(x))$ as in Lemma~\ref{lem:detour}.
Note that
\begin{equation}\label{eq:Phi-u}
 \Phi(\mu)=\mu(f) \qquad \text{for all $\mu\in\lipfree{M}$.}
\end{equation}
This holds first for finite sums of evaluations by definition,
and then for all $\mu$ by norm continuity and density.

Fix distinct $p,q\in M$.
Using Lemma~\ref{lem:arc}, choose an injective rectifiable
path $\gamma:[0,L]\to M$ from $p$ to $q$, parametrized
by arclength.
Let $S$ be the set in \eqref{eq:cutset}, and take the
projection $P$ from Lemma~\ref{lem:cut-projection2}.
Applying its predual inclusion to both $Y$ and the
canonical predual $\lipfree{M}$ gives
\[
 P^*\mu\in Y\cap\lipfree{M}
 \qquad \text{for all $\mu\in\lipfree{M}$.}
\]
Thus \eqref{eq:Phi-u} and the fact that $\Phi$ vanishes
on $Y$ imply
\[
 \mu(Pf)=(P^*\mu)(f)=\Phi(P^*\mu)=0
 \qquad \text{for all $\mu\in\lipfree{M}$.}
\]
Since $\lipfree{M}$ separates points of $\Lip_0(M)$,
we conclude that $Pf=0$.
Now \eqref{eq:cut-projection-derivative} in Lemma~\ref{lem:cut-projection2} implies that $(f\circ\gamma)'(t)=0$ for almost every $t\in [0,L]$ with $\gamma(t)\in S$, and Lemma~\ref{lem:detour} gives the same conclusion for $\gamma(t)\notin S$. It follows that $f\circ\gamma$ is constant, and $f(p)=f(q)$.

Since $p,q$ were arbitrary and $f(0)=0$, we have just shown that $f=0$,
and \eqref{eq:Phi-u} proves that $\Phi$
vanishes on $\lipfree{M}$, as required.
\end{proof}

\section{Proof for all metric spaces}
\label{sec:general}

We now extend our main result to the general case. We will do so by embedding our metric space $M$ into a length space $N$ in such a way that preduals of their Lipschitz spaces are compatible, in the sense that the restriction operator is weak$^*$-to-weak$^*$ continuous. Strong uniqueness for $M$ is then reduced to strong uniqueness for $N$, which has already been established.

We first describe a general enlargement construction, which is standard. For the remainder of this section, we fix a complete metric space $M$. Let $\mathscr E$
be a family of ordered pairs $e=(p_e,q_e)$ of distinct
points of $M$. Put $\ell_e=d(p_e,q_e)$ and form the
disjoint union
\[
 X=M\sqcup\bigsqcup_{e\in\mathscr E}
       \bigl(\{e\}\times[0,\ell_e]\bigr).
\]
Following \cite[I.5.18]{BH}, equip $X$ with the extended
metric $d_X$ which agrees with the given metric on each
piece and is $+\infty$ between different pieces.
Let $\sim$ be the equivalence relation generated by
$(e,0)\sim p_e$ and $(e,\ell_e)\sim q_e$, and write
$\pi:X\to N:=X/{\sim}$ for the quotient map.

Equip $N$ with the quotient pseudometric of
\cite[I.5.19]{BH}:
\[
 d_N(x,y)=\inf\left\{
 \sum_{i=1}^n d_X(a_i,b_i):
 \begin{array}{l}
 n\geq1,\quad a_i,b_i\in X,\\
 \pi(a_1)=x,\quad \pi(b_n)=y,\\
 \pi(b_i)=\pi(a_{i+1})\quad(1\leq i<n)
 \end{array}
 \right\}.
\]
The infimum is finite, since each interval is attached
to $M$. We identify $M$ with its image in $N$, denote
the image of $\{e\}\times[0,\ell_e]$ by $I_e$, and write
$(e,t)$ for its point with parameter $t$.
The following lemma justifies these identifications
metrically and records the properties used below.

\begin{lemma}\label{lem:gluing-metric2}
The above construction has the following properties.
\begin{enumerate}[label=\textup{(\roman*)}]
\item\label{item:gluing-complete}
The function $d_N$ is a complete metric. The natural maps
from $M$ and $[0,\ell_e]$ into $N$ are isometric embeddings.

\item\label{item:gluing-endpoint}
For $z=(e,t)$ with $0<t<\ell_e$ and $x\in N\setminus I_e$,
\begin{equation}\label{eq:edge-exit-distance}
 d_N(z,x)=
 \min\{t+d_N(p_e,x),\,\ell_e-t+d_N(q_e,x)\}.
\end{equation}
In particular, there is an endpoint $a\in\{p_e,q_e\}$
such that $d_N(z,x)=d_N(z,a)+d_N(a,x)$.

\item\label{item:gluing-Lipschitz}
For $C\geq0$, a function $u:N\to\RR$ is $C$-Lipschitz
if and only if its restriction to $M$ and its restriction
to every $I_e$ are $C$-Lipschitz.

\item\label{item:gluing-diameter}
We have $\diam N\leq2\diam M$. In particular, if $M$ is bounded then $N$ is bounded.
\end{enumerate}
\end{lemma}

\begin{proof}
If $\mathscr E=\varnothing$, all assertions are immediate.
Otherwise, first glue $M$ and $[0,\ell_e]$ along the
closed isometric subsets $\{p_e,q_e\}$ and
$\{0,\ell_e\}$, obtaining $M_e$. Then glue all the spaces
$M_e$ along their common copy of $M$.
By \cite[Lemma~I.5.24 and Exercise~I.5.25(1)]{BH},
both constructions give complete metric spaces with
isometric inclusions of their pieces; the common copy
of $M$ is closed because it is complete.
Denote the metric obtained on $N$ in this way by
$\widetilde d$.

The distance formula in \cite[Lemma~I.5.24]{BH}, applied
at the first stage, gives
\begin{equation}\label{eq:gluing-first-stage}
 d_{M_e}((e,t),a)
 =\min\{t+d(p_e,a),\,\ell_e-t+d(q_e,a)\}
 \qquad(a\in M).
\end{equation}
At the second stage, it gives
\begin{equation}\label{eq:gluing-second-stage}
 \widetilde d(x,y)
 =\inf_{a\in M}
   \bigl(d_{M_e}(x,a)+d_{M_f}(a,y)\bigr)
 \qquad(x\in M_e,\ y\in M_f,\ e\ne f).
\end{equation}

We first verify that $\widetilde d=d_N$.
For any $u,v\in M_e$, we have
$d_N(u,v)\leq d_{M_e}(u,v)$.
Indeed, if both points lie in the same original piece,
a quotient chain consisting of one pair gives this
inequality. Otherwise, \eqref{eq:gluing-first-stage}
applies, and each of its two expressions is the sum
$\sum_i d_X(a_i,b_i)$ for a quotient chain through
the corresponding endpoint.
Consequently, if $x,y\in M_e$, then
$d_N(x,y)\leq d_{M_e}(x,y)=\widetilde d(x,y)$.
If $x\in M_e$ and $y\in M_f$ with $e\ne f$, then
\[
 d_N(x,y)
 \leq d_N(x,a)+d_N(a,y)
 \leq d_{M_e}(x,a)+d_{M_f}(a,y)
 \qquad(a\in M).
\]
Taking the infimum in \eqref{eq:gluing-second-stage}
again gives $d_N(x,y)\leq\widetilde d(x,y)$.

Conversely, $\widetilde d$ agrees with the given metric
on every original piece. For every quotient chain
with finite $\sum_i d_X(a_i,b_i)$, its triangle
inequality therefore gives
\[
 \widetilde d(\pi(a_1),\pi(b_n))
 \leq\sum_{i=1}^n d_X(a_i,b_i).
\]
Taking the infimum proves $\widetilde d\leq d_N$.
Thus $\widetilde d=d_N$, proving \textup{(i)}.

To obtain \textup{(ii)}, let $z=(e,t)$ and
$x\in N\setminus I_e$. If $x\in M$, the assertion
is \eqref{eq:gluing-first-stage}.
Otherwise, $x$ lies in the interior of some $I_f$
with $f\ne e$. Substituting \eqref{eq:gluing-first-stage} into
\eqref{eq:gluing-second-stage} gives
\begin{align*}
 d_N(z,x)
 &=\inf_{a\in M}\min\bigl\{
     t+d(p_e,a)+d_{M_f}(a,x),\,\ell_e-t+d(q_e,a)+d_{M_f}(a,x)
   \bigr\}\\
 &=\min\Bigl\{
     t+\inf_{a\in M}\bigl(d(p_e,a)+d_{M_f}(a,x)\bigr),\,\ell_e-t+\inf_{a\in M}
       \bigl(d(q_e,a)+d_{M_f}(a,x)\bigr)
   \Bigr\}.
\end{align*}
Notice further that for $b\in\{p_e,q_e\}$, we have
\[
 \inf_{a\in M}
 \bigl(d(b,a)+d_{M_f}(a,x)\bigr)
 =d_{M_f}(b,x)=d_N(b,x),
\]
where the first equality follows from the triangle
inequality, with equality attained at $a=b$. Consequently,
\[
 d_N(z,x)=
 \min\{t+d_N(p_e,x),\,\ell_e-t+d_N(q_e,x)\},
\]
which is \eqref{eq:edge-exit-distance}.

For \textup{(iii)}, only the reverse implication needs
proof. In a quotient chain with finite
$\sum_i d_X(a_i,b_i)$, each pair $a_i,b_i$ lies in
one original piece. Since
$\pi(b_i)=\pi(a_{i+1})$, the assumed Lipschitz bounds
give
\[
 |u(x)-u(y)|
 \leq\sum_{i=1}^n
      |u(\pi(a_i))-u(\pi(b_i))|
 \leq C\sum_{i=1}^n d_X(a_i,b_i).
\]
Taking the infimum proves that $u$ is $C$-Lipschitz.

Finally, part \textup{(iv)} is clear if $M$ is unbounded. If $D=\diam M<\infty$, every point of $N$
is within distance $D/2$ of a point of $M$: for a point
of an attached interval, choose its nearer endpoint.
The triangle inequality gives
$\diam N\leq D/2+D+D/2=2D$, proving \textup{(iv)}.
\end{proof}

\begin{proposition}\label{prop:length-enlargement2}
If, in the construction above, we take
\[
 \mathscr E=\{(p,q)\in M\times M:
 p\ne q\text{ and }(p,q)\text{ fails property~(Z) in }M\}.
\]
then $N$ is a complete length space.
\end{proposition}

\begin{proof}
Completeness follows from Lemma~\ref{lem:gluing-metric2}\textup{(i)}.
By Theorem~\ref{thm:AMC}, it remains to prove that $N$
has property~(Z).

Let $p\neq q\in M$. If $(p,q)$ has property~(Z) in $M$ then it still does in $N$, with the same witnesses. Otherwise, $(p,q)\in\mathscr E$ and its attached interval $I_{(p,q)}$ contains a midpoint
$w\notin\{p,q\}$ satisfying
\[
 d_N(p,w)+d_N(w,q)=d_N(p,q).
\]
Thus $(p,q)$ has property~(Z) in $N$ as well.

Now consider a pair $(x,y)$ with $x\notin M$, say $x=(e,t)$. If $y\in I_e\setminus\set{x}$, then any point of $I_e$ strictly between $x$ and $y$ witnesses property~(Z) of $(x,y)$. Otherwise, Lemma~\ref{lem:gluing-metric2}\textup{(ii)}
gives an endpoint $a\in\{p_e,q_e\}$ such that
\[
 d_N(x,a)+d_N(a,y)=d_N(x,y).
\]
Since $x\notin M$ and $y\notin I_e$,
we have $a\notin\{x,y\}$. Hence $a$ itself is a witness
for property~(Z) of $(x,y)$, completing the proof.
\end{proof}

Finally, we establish the required compatibility for preduals. This proposition applies to any family $\mathscr E$, not just the one from Proposition~\ref{prop:length-enlargement2}, as long as the hypothesis is satisfied.

\begin{proposition}\label{prop:compatible-predual}
Let $Y$ be a concrete predual of $\Lip_0(M)$ and suppose that $m_{p_eq_e}\in Y$ for all $e\in\mathscr E$.
Then $\Lip_0(N)$ has a concrete predual $Y_N$ for which the
restriction $R:\Lip_0(N)\to\Lip_0(M)$
satisfies $R^*Y\subset Y_N$. In particular, $R$ is continuous
from $\sigma(\Lip_0(N),Y_N)$ to $\sigma(\Lip_0(M),Y)$.
\end{proposition}

\begin{proof}
Form the Banach space
\[
 \mathcal H=Y\oplus_1
 \left(\bigoplus_{e\in\mathscr E}L_1([0,\ell_e])\right)_{\ell_1}.
\]
Since $Y^*=\Lip_0(M)$ canonically, its dual is
\[
 \mathcal H^*=\Lip_0(M)\oplus_\infty
 \left(\bigoplus_{e\in\mathscr E}L_\infty([0,\ell_e])\right)_{\ell_\infty},
\]
with pairing
\[
 \bigl\langle(f,(g_e)),(y,(a_e))\bigr\rangle
 =y(f)+\sum_{e\in\mathscr E}\int_0^{\ell_e}a_e(t)g_e(t)\,dt.
\]
All such sums are absolutely convergent; a member of an $\ell_1$
sum has at most countably many nonzero coordinates.

Let us see that $\Lip_0(N)$ is isometric to the subspace
\begin{equation}\label{eq:compatible-W}
 \mathcal W=
 \left\{(f,(g_e))\in\mathcal H^*:
     \int_0^{\ell_e}g_e(t)\,dt=f(q_e)-f(p_e)
     \text{ for every }e\in\mathscr E\right\}.
\end{equation}
Define an operator $\Theta:\Lip_0(N)\to\mathcal W$ by $\Theta u=(u|_M,(u_e'))$, where we write $u_e(t)=u(e,t)$ for $e\in\mathscr E$, $t\in[0,\ell_e]$. Then Lemma~\ref{lem:gluing-metric2}\textup{(iii)} gives
\begin{equation}\label{eq:norm-estimation}
 \lipnorm{u}
 =\max\left\{\lipnorm{u|_M}, \sup_{e\in\mathscr E}\lipnorm{u|_{I_e}}\right\}
 =\max\left\{\lipnorm{u|_M}, \sup_{e\in\mathscr E}\norm{u_e'}_\infty\right\}
 =\norm{\Theta u}_{\mathcal H^*},
\end{equation}
where the supremum is understood as zero if $\mathscr E=\varnothing$. Thus $\Theta$ is an isometry.
Now let $(f,(g_e))\in\mathcal W$ and define a function $u:N\to\RR$ by
\[
 u|_M=f,\qquad
 u(e,t)=f(p_e)+\int_0^t g_e(s)\,ds.
\]
The constraints make the values at endpoints $p_e,q_e$ consistent. It is clear that $\Theta u=(f,(g_e))$, and the computation \eqref{eq:norm-estimation} shows that $u$ is Lipschitz. Moreover $u(0)=f(0)=0$, so $u$ really belongs to $\Lip_0(N)$. Thus $\Theta$ is surjective.

Since $\delta(p_e)-\delta(q_e)=\ell_e m_{p_eq_e}\in Y$ by assumption, the vector
\[
 \nu_e=(\delta(p_e)-\delta(q_e),(0,\ldots,\one_{[0,\ell_e]},\ldots))
\]
belongs to $\mathcal H$ for all $e\in\mathscr E$. Set
$\mathcal V=\overline{\lspan}\{\nu_e:e\in\mathscr E\}$, with norm
closure in $\mathcal H$. Equation~\eqref{eq:compatible-W} says
exactly that $\mathcal W=\mathcal V^\perp$, consequently
$(\mathcal H/\mathcal V)^*=\mathcal W$ isometrically.
Define
\[
 \mathcal J:\mathcal H/\mathcal V\longrightarrow\Lip_0(N)^*,
 \qquad
 [\mathcal J(h+\mathcal V)](u)=\langle\Theta u,h\rangle.
\]
This is well defined because $\Theta u\in\mathcal V^\perp$.
Since $\Theta$ maps the unit ball of $\Lip_0(N)$ onto
the unit ball of $\mathcal V^\perp$, the canonical duality
$(\mathcal H/\mathcal V)^*=\mathcal V^\perp$ gives that $\mathcal J$ is a linear isometry, and its range
$Y_N:=\mathcal J(\mathcal H/\mathcal V)$ is a closed
subspace of $\Lip_0(N)^*$.

To verify that $Y_N$ is a concrete predual of $\Lip_0(N)$, consider the
canonical evaluation map
\[
 \kappa:\Lip_0(N)\longrightarrow Y_N^*,
 \qquad [\kappa(u)](\lambda)=\lambda(u).
\]
Regarded as a map onto $Y_N$, the isometry $\mathcal J$
has a surjective isometric adjoint
$\mathcal J^*:Y_N^*\to(\mathcal H/\mathcal V)^*=\mathcal W$.
For every $h\in\mathcal H$,
\[
 [\mathcal J^*\kappa(u)](h+\mathcal V)
 =[\mathcal J(h+\mathcal V)](u)
 =\langle\Theta u,h\rangle.
\]
Hence $\mathcal J^*\kappa=\Theta$. Since both
$\mathcal J^*$ and $\Theta$ are surjective linear
isometries, so is $\kappa$. Therefore $Y_N$ is a
concrete predual of $\Lip_0(N)$.

Finally, for $y\in Y$ and $u\in\Lip_0(N)$,
\[
 [\mathcal J((y,0)+\mathcal V)](u)
 =\langle\Theta u,(y,0)\rangle
 =y(Ru)
 =(R^*y)(u).
\]
Thus $R^*y=\mathcal J((y,0)+\mathcal V)\in Y_N$,
proving $R^*Y\subset Y_N$. In particular, every
$y\circ R$ is $\sigma(\Lip_0(N),Y_N)$-continuous.
By the definition of $\sigma(\Lip_0(M),Y)$, this proves
the asserted continuity of $R$.
\end{proof}

With these ingredients, we can finally prove the general case of our main theorem.

\begin{proof}[Proof of Theorem~\ref{thm:main}]
Fix a concrete predual $Y\subset\Lip_0(M)^*$.
Let $\mathscr E$ consist of all ordered pairs of distinct
points of $M$ failing property~(Z), and let $N$ be the
complete length enlargement given by
Proposition~\ref{prop:length-enlargement2}, with the same
base point as $M$. For every $e=(p_e,q_e)\in\mathscr E$,
Theorem~\ref{thm:GLPRZ} and Lemma~\ref{lem:exposed} give $m_{p_eq_e}\in Y$.
Proposition~\ref{prop:compatible-predual} therefore
provides a concrete predual $Y_N$ of $\Lip_0(N)$ such
that $R^*Y\subset Y_N$, where
$R:\Lip_0(N)\to\Lip_0(M)$ is the restriction to $M$.
By Theorem~\ref{thm:length}, $Y_N=\lipfree{N}$, so
\[
 R^*Y\subset\lipfree{N}.
\]

Let $j:\lipfree{M}\to\lipfree{N}$ be the canonical
isometric inclusion. Since $R=j^*$, the standard range--kernel identity gives
\[
 (\ker R)_\perp
 =(\ker j^*)_\perp
 =\overline{\operatorname{ran}j}^{\norm{\cdot}}
 =j\lipfree{M},
\]
where the preannihilator is taken in $\lipfree{N}$,
and the last equality holds because $j$ is an isometric
embedding. Every element of $R^*Y$ belongs to
$\lipfree{N}$ and annihilates $\ker R$. Consequently,
\[
 R^*Y\subset(\ker R)_\perp=j\lipfree{M}.
\]

To conclude, fix $y\in Y$ and choose $\mu\in\lipfree{M}$
such that $R^*y=j\mu$. For any $f\in\Lip_0(M)$,
McShane' theorem yields $u\in\Lip_0(N)$ with $Ru=f$.
Then
\[
 y(f)=(R^*y)(u)=(j\mu)(u)=\mu(Ru)=\mu(f).
\]
Thus $y=\mu\in\lipfree{M}$, proving $Y\subset\lipfree{M}$.
Since concrete preduals cannot properly contain each other, we get $Y=\lipfree{M}$, and Fact~\ref{fact:concrete-preduals} yields strong uniqueness.
\end{proof}

\appendix
\section{A counterexample to the codimension-one lemma}
\label{sec:counterexample}

Here, we prove that \cite[Lemma~3.1]{Weaver18} is false by constructing an explicit counterexample.

\begin{theorem}
\label{thm:counterexample}
There exists a Banach space $X$ such that
\begin{enumerate}[label={\upshape{(\alph*)}}]
\item $X$ is isomorphic to $\ell_1$,
\item $X$ has a strongly unique predual, and
\item $X$ has a weak$^*$ closed, $1$-complemented, $1$-codimensional subspace isometric to $\ell_1$.
\end{enumerate}
\end{theorem}

\begin{proof}
We start by defining the predual of $X$, which will be an equivalent renorming of $c_0$. We denote by $e_n$ the $n$-th canonical basis vector of $c_0$, and set
$$
\Gamma = \set{e_1+e_n,e_1-e_n,-e_1+e_n,-e_1-e_n \,:\, n\geq 2}
$$
and $D=\set{y\in B_{c_0} \,:\, y_1=0}$.
Then we let $\norm{\cdot}_Y$ be the norm on $c_0$ whose unit ball is the set
$$
B_Y = \cl{\conv}\pare{\Gamma \cup D} ,
$$
and put $Y=(c_0,\norm{\cdot}_Y)$. To see that this does in fact define an equivalent renorming of $c_0$, it suffices to verify that
\begin{equation}
\label{eq:Y eqv norm}
\tfrac{1}{2}B_{c_0} \subset B_Y \subset B_{c_0} .
\end{equation}
The rightmost inclusion in \eqref{eq:Y eqv norm} is clear as $\Gamma\subset B_{c_0}$. For the other one, note that
$$
e_1 = \tfrac{1}{2}(e_1+e_2) + \tfrac{1}{2}(e_1-e_2) \in \conv(\Gamma)
$$
and similarly $-e_1\in\conv(\Gamma)$, hence $te_1\in\conv(\Gamma)$ for all $t\in [-1,1]$. Thus, for any $y\in B_{c_0}$,
$$
\tfrac{1}{2}y = \tfrac{1}{2}y_1e_1 + \tfrac{1}{2}(y-y_1e_1) \in B_Y .
$$

Now let $X=Y^*$. Then $X$ is an equivalent renorming of $\ell_1$. In fact, we may describe its norm $\norm{\cdot}_X$ explicitly: for any $x\in\ell_1$, we claim that
\begin{equation}
\label{eq:X eqv norm}
\norm{x}_X = \max\set{ \sum_{n\geq 2}\abs{x_n} , \abs{x_1}+\max_{n\geq 2}\abs{x_n} } .
\end{equation}
Indeed, note that
$$
\norm{x}_X = \sup_{y\in B_Y}\duality{x,y} = \sup_{y\in B_Y}\sum_{n=1}^\infty x_ny_n = \sup\set{ \sum_{n=1}^\infty x_ny_n : y\in\Gamma \cup D } .
$$
If $y\in D$, the sum is bounded by $\sum_{n\geq 2}\abs{x_n}$ and this value can be approached by taking $y_n=\mathrm{sign}(x_n)$ for $2\leq n\leq N$ and $y_n=0$ for $n>N$. Similarly, for $y\in\set{e_1\pm e_n,-e_1\pm e_n}$, $n\geq 2$ the maximum value of $\duality{x,y}$ is $\abs{x_1}+\abs{x_n}$, and taking the maximum over $n$ yields \eqref{eq:X eqv norm}.

Let $V$ be the subspace of $X$ given by
$$
V = \set{x\in X : x_1=0} .
$$
It is clear that $V$ has codimension $1$ in $X$, and moreover $V=\ker(e_1)$ with $e_1\in Y$ so it is weak$^*$ closed. Applying \eqref{eq:X eqv norm} we see that $\norm{x}_X=\norm{x}_1$ for any $x\in V$, so $V$ is isometric to $\ell_1$. Finally, note that the mapping $P:X\to V$ given by $Px=x-x_1e_1^*$ (where $e_n^*$ denotes the $n$-th canonical basis vector of $\ell_1$) is a linear projection onto $V$, and $\norm{Px}_X = \norm{Px}_1 = \sum_{n\geq 2}\abs{x_n} \leq \norm{x}_X$, so $\norm{P}=1$. Thus $V$ is $1$-complemented. This shows that $X$ and $V$ satisfy items (a) and (c).

In order to establish condition (b), we prove that every element of $\Gamma$ is a strongly exposed point of $Y$. Fix $n\geq 2$ and set $y_n=e_1+e_n\in\Gamma$ and $x_n=\frac{1}{3}e_1^*+\frac{2}{3}e_n^*\in B_X$. It is clear that $\duality{x_n,y_n}=1$. We also have $\duality{x_n,y}\leq\frac{1}{3}$ for any $y\in\Gamma\setminus\set{y_n}$ and $\duality{x_n,y}\leq\frac{2}{3}$ for any $y\in D$. It follows that $\duality{x_n,c}\leq\frac{2}{3}$ for any $c$ belonging to the set
$$
C = \wscl{\conv}( \Gamma\cup D\setminus\set{y_n} ) \subset B_{X^*} .
$$
Note that $B_{X^*}=\conv(C\cup\set{y_n})$. Indeed, $C$ is weak$^*$ compact, therefore so is $\conv(C\cup\set{y_n})$, and the latter contains $\Gamma$ and $D$, so it contains $B_Y$ and thus $B_{X^*}$ by Goldstine's theorem. Therefore, any element $z\in B_{X^*}$ can be written as $z=ty_n+(1-t)c$ for some $t\in [0,1]$ and $c\in C$. For such an expression, we have
$$
\duality{x_n,z} = t\duality{x_n,y_n} + (1-t)\duality{x_n,c} \leq t+\tfrac{2}{3}(1-t) = 1-\tfrac{1}{3}(1-t)
$$
hence
$$
1-t \leq 3(1-\duality{x_n,z}) .
$$
It follows that
$$
\norm{z-y_n}_{X^*} = (1-t)\norm{y_n-c}_{X^*} \leq 2(1-t) \leq 6(1-\duality{x_n,z}) .
$$
In particular, if $\duality{x_n,z}\to 1$ then $\norm{z-y_n}\to 0$. That is, $x_n$ strongly exposes $y_n$. A similar argument shows that $y'_n=e_1-e_n$ is strongly exposed by $x'_n=\frac{1}{3}e_1^*-\frac{2}{3}e_n^*\in B_X$, and $-y_n,-y'_n$ are also strongly exposed points by symmetry.

Note that
$$
e_n = \tfrac{1}{2}(e_1+e_n) + \tfrac{1}{2}(-e_1+e_n) \in \conv(\Gamma)
$$
for all $n\geq 2$ and we already saw that $e_1\in\conv(\Gamma)$ as well, thus $\cl{\lspan}(\Gamma)=Y$. Lemma~\ref{lem:exposed} now implies that every concrete predual of $X$ contains $\Gamma$ and hence $Y$. By Fact~\ref{fact:concrete-preduals}, $Y$ is the strongly unique predual of $X$. This ends the proof.
\end{proof}

\section*{Acknowledgments}
F.~Vico was supported by the Simons Foundation under project
SFI-FI-CCM-Grant-00030182.

\end{document}